\documentclass[10pt]{amsart}
\usepackage{amsmath, amssymb, latexsym, enumerate}
\usepackage{graphicx}

\numberwithin{equation}{section}
\newtheorem{thm}{Theorem}[section]
\newtheorem{pro}[thm]{Proposition}
\newtheorem{cor}[thm]{Corollary}
\newtheorem{lem}[thm]{Lemma}
\theoremstyle{definition}
\newtheorem{exa}[thm]{Example}

\newcommand{\bZ}{\ensuremath{\mathbb Z}}
\newcommand{\cb}{\ensuremath{\mathcal B}}
\newcommand{\cc}{\ensuremath{\mathcal C}}
\newcommand{\cf}{\ensuremath{\mathcal F}}
\newcommand{\ci}{\ensuremath{\mathcal I}}
\newcommand{\cj}{\ensuremath{\mathcal J}}
\newcommand{\cm}{\ensuremath{\mathcal M}}
\newcommand{\cs}{\ensuremath{\mathcal S}}
\newcommand{\cu}{\ensuremath{\mathcal U}}
\newcommand{\cz}{\ensuremath{\mathcal Z}}
\newcommand{\GF}{\hbox{\rm GF}}
\newcommand{\PG}{\hbox{\rm PG}}
\newcommand{\del}{\backslash} 
\newcommand{\sands}{\quad\text{and}\quad}
\newcommand{\co}{\,\colon\,}
\newcommand{\rta}{\rightarrow}
\renewcommand{\=}{\,=\,}
\newcommand{\sub}[1]{_{\text{${\scriptscriptstyle #1}$}}}
\newcommand{\subb}[2]{_{\text{${\scriptscriptstyle #1,#2}$}}}
\newcommand{\les}{\unlhd}

\newcommand{\rank}[1]{\ensuremath{r_{\scriptscriptstyle #1}}}
\newcommand{\rk}[2]{\ensuremath{r_{\scriptscriptstyle #1}(#2)}}
\renewcommand{\r}[1]{\ensuremath{r(#1)}}
\newcommand{\rkm}[1]{\ensuremath{r_{\scriptscriptstyle M}(#1)}}
\newcommand{\rkn}[1]{\ensuremath{r_{\scriptscriptstyle N}(#1)}}
\newcommand{\rkl}[1]{\ensuremath{r_{\scriptscriptstyle L}(#1)}}
\newcommand{\rkp}[1]{\ensuremath{r_{\scriptscriptstyle P}(#1)}}
\newcommand{\cl}[2]{\ensuremath{\text{\rm cl}_{\scriptscriptstyle #1}(#2)}}
\newcommand{\cll}[1]{\ensuremath{\text{\rm cl}_{\scriptscriptstyle L}(#1)}}
\newcommand{\clm}[1]{\ensuremath{\text{\rm cl}_{\scriptscriptstyle M}(#1)}}
\newcommand{\cln}[1]{\ensuremath{\text{\rm cl}_{\scriptscriptstyle N}(#1)}}

\newcommand{\splice}[2]{\ensuremath{\text{Sp}(#1,#2)}}
\newcommand{\fsp}{\mathbin{\Join}}
\newcommand{\frp}{\mathbin{\Box}}
\newcommand{\fssp}{\,\fsp\,}
\newcommand{\fs}[1]{\ensuremath{\mathcal{FS} (#1)}}

\newcommand{\rlift}[2]{\ensuremath{\text
{\rm L$_{\scriptscriptstyle #2}^{#1}$}}}

\newcommand{\rtrun}[2]{\ensuremath{\text
{\rm T$_{\scriptscriptstyle #2}^{#1}$}}}
\newcommand{\isthm}[1]{\ensuremath{\text{\rm Isth}(#1)}}
\newcommand{\loops}[1]{\ensuremath{\text{\rm Loop}(#1)}}
\newcommand{\ist}[1]{\ensuremath{#1_1}}

\newcommand{\loo}[1]{\ensuremath{#1_0}}
\newcommand{\higgs}[3]{\ensuremath{L^{#1}_{\scriptscriptstyle #3, #2}}}
\newcommand{\ideal}[1]{\ensuremath{\mathcal J\sub{#1}}}
\newcommand{\idea}[3]{\ensuremath{\mathcal J\sub{#1}(#2,#3)}}
\newcommand{\ide}[4]{\ensuremath{\mathcal J\sub{#1}^{#2}(#3,#4)}}

\begin{document}
\title[Splicing Matroids]
{Splicing Matroids}
\author[J.~Bonin]
       {Joseph E.~Bonin}
\address 
{Department of Mathematics\\ The George Washington University\\
Washington, D.C. 20052, USA} \email{jbonin@gwu.edu, wschmitt@gwu.edu}
\author[W.~Schmitt]{William R. Schmitt}
\date{\today}
\bibliographystyle{abbrv}

\begin{abstract}
  We introduce and study a natural variant of matroid amalgams.  For
  matroids $M(A)$ and $N(B)$ with $M.(A\cap B)=N|(A\cap B)$, we define
  a splice of $M$ and $N$ to be a matroid $L$ on $A\cup B$ with
  $L|A=M$ and $L.B=N$.  We show that splices exist for each such pair
  of matroids $M$ and $N$; furthermore, there is a freest splice of
  $M$ and $N$, which we call the free splice.  We characterize when a
  matroid $L(A\cup B)$ is the free splice of $L|A$ and $L.B$.  We
  study minors of free splices and the interaction between free
  splice and several other matroid operations.  
Although free splice is not an associative operation, we prove a
  weakened counterpart of associativity that holds in general and we
  characterize the triples for which associativity holds.  We also study 
free splice as it relates to various classes of matroids.
\end{abstract}

\maketitle 
\begin{center}
  \emph{In memory of Tom Brylawski}
\end{center}
\vspace{.25in}

\section{Introduction}

Considering a variation of the well-known problem of matroid amalgams
led us to the new matroid construction that we develop in this paper.
This construction turns out to have many attractive properties that have
no counterparts for amalgams.

We first briefly recall matroid amalgams; for an excellent account of
this topic, see \cite[Section 12.4]{ox:mt}.  Given matroids $M(A)$ and
$N(B)$ with $M|(A\cap B)=N|(A\cap B)$, an amalgam of $M$ and $N$ is a
matroid $L$ on $A\cup B$ with $L|A=M$ and $L|B=N$.  The most familiar
special case is the generalized parallel connection, which glues two
matroids together, in the freest possible way, along what is
essentially a modular flat of one of the matroids~\cite{br:mcc,ox:mt}.
Apart from this atypically nice case, though, two matroids may have
incompatible structure and thus have no amalgam.  Also, pairs of matroids
that have amalgams may have no freest amalgam.

Now consider a variation on this theme.  Since deletion and
contraction commute, if we are given a matroid $L(A\cup B)$, then the
restriction $M=L|A$ and contraction $N=L.B$ have the property that
$M.(A\cap B)=N|(A\cap B)$.  Hence, starting with matroids $M(A)$ and
$N(B)$ satisfying $M.(A\cap B)=N|(A\cap B)$, it is natural to ask if
there is a matroid $L(A\cup B)$ such that $L|A=M$ and $L.B=N$;
we call such a matroid $L$ a splice of $M$ and $N$.
We show that, in contrast to the situation for amalgams, splices of 
such $M$ and $N$ always exist and, furthermore, there is always a 
freest splice; we call this matroid the free splice of $M$ and $N$
and denote it by  $M\fsp N$.
If $A\cap B=\emptyset$, then the free
splice is the free product~\cite{crsc:fpm, crsc:uft}.

In this paper we introduce matroid splices and study the free splice
operation in depth.  Consistent with a common phenomenon in matroid
theory, there are many ways that one may approach the free splice; we
use the Higgs lift to provide an efficient framework for developing
these various approaches.  A key component of Section
\ref{sec:prelim}, which treats preliminary matters, is a unified
treatment of cryptomorphic formulations of the Higgs lift; this
section also treats minors of Higgs lifts and the behavior of the
Higgs lift with respect to the weak and strong orders.  In Section
\ref{sec:splice}, we define splices in general and show that they form
a filter in a certain natural suborder of the weak order.  In the main
part of this section, we define the free splice as a particular Higgs
lift, we show that it is indeed the freest splice, we show that it
preserves the weak and strong orders, and we provide cryptomorphic
formulations of this operation.  One attractive result is the relation
between free splice and duality: $(M\fsp N)^*=N^*\fsp M^*$.  Using
cyclic flats, the basic question of when a matroid $L(A\cup B)$ is the
free splice of $L|A$ and $L.B$ is answered in Section
\ref{sec:factor}; a corollary of this work is that a matroid is
irreducible with respect to free splice if and only if, for each
ordered pair $x,y$ of distinct elements, there is cyclic flat that
contains $x$ but not $y$.  In Section \ref{sec:inter} we study the
interaction between free splice and several other matroid operations;
we show, for instance, that direct sums and generalized parallel
connections of irreducible matroids are irreducible, we study minors
of free splices, and we show that the free splice can be realized as
the intersection of certain free products.  Like the free product
operation, free splice is noncommutative but, in contrast to both free
product and direct sum, it is also nonassociative. In Section
\ref{sec:assoc} we prove a weakened version of associativity for free
splice and characterize triples of matroids for which associativity
holds.  We also describe the (very special) conditions under which
free splice is commutative.  In contrast to free product, free splice
does not preserve many commonly studied matroid properties (e.g.,
being representable, transversal, or base-orderable); Section
\ref{sec:class} contains examples addressing these points and gives a
necessary condition for a minor-closed class of matroids to be closed
under free splice.

We assume the reader is familiar with basic matroid theory.
Background on particular topics can be found in~\cite{ox:mt,we:mt,wh:tom}.

\section{Preliminaries}\label{sec:prelim}

The Higgs lift provides an efficient framework for defining the free
splice and treating its cryptomorphic formulations.  To pave the way,
the Higgs lift is reviewed in the main part of this section and a
unified account of its cryptomorphic formulations is presented.  (Many
of these formulations are known; others, such as that using cyclic
flats, as well as the unifying perspective itself, may be new.)
Similarly, results in this section on minors of Higgs lifts and the
fact that the Higgs lift preserves the weak and strong orders will be
used when treating corresponding results about free splice.  A
short section on submodular functions precedes and prepares the way
for the work on the Higgs lift.

\subsection{Notation and terminology}

A set $X$ in a matroid $M(E)$ is \emph{cyclic} 
if $X$ is a union of circuits
of $M$, that is, if $M|X$ has no isthmuses.  We write $M^*$ for the dual
of $M$ and denote by $\ci (M)$, $\cb (M)$, $\cc (M)$, $\cs (M)$, $\cf
(M)$, and $\cz (M)$, respectively, the collections of independent
sets, bases, circuits, spanning sets, flats, and cyclic flats of $M$.
For any family \cu\ of subsets of $E$, we write $\cu^c$ for the family
$\{E-X\co X\in\cu\}$.  Since $E-X$ is a flat of $M^*$ if and only if
$M^*\!.\,X$ has no loops, that is, $M|X$ has no isthmuses, it follows that
$\cf (M^*)^c$ is the collection of cyclic sets of $M$, so $\cz (M)
=\cf (M)\cap\cf (M^*)^c$.  We denote the rank function of $M$ by \rank
M\ and write \r M\ for $\rkm E$.  We write \isthm M\ and \loops M\ for the 
sets of all isthmuses and loops of $M$ and denote by $I(E)$ the 
free matroid on $E$; thus, $\isthm{I(E)}=\loops{I^*(E)}=E$.

For any finite set $E$ we denote by $\cm (E)$ the set of all matroids
on $E$.  The \emph{weak order}
on $\cm (E)$, denoted by $\leq$, is given by
$N\leq M$ if and only if $\ci (N)\subseteq\ci (M)$ or, equivalently,
$\rkn X\leq\rkm X$ for all $X\subseteq E$.  We say that $M$ is
\emph{freer} than $N$ if $N\leq M$.
Note that if $\r M=\r N$,
then $N\leq M$ if and only if $\cb (N)\subseteq \cb (M)$; thus in
this case, $N\leq M$ if and only if $N^*\leq M^*$.
The \emph{strong order} on the set of matroids $\cm (E)$ is defined by
setting $N\les M$ if and only if $\cf (N)\subseteq\cf (M)$.  This
relation is also described by saying that $N$ is a \emph{quotient} of $M$, or
that $M$ is a \emph{lift} of $N$.  It is well known that $N\les M$ if and
only if $\rkn Y -\rkn X\leq\rkm Y -\rkm X$ for all $X\subseteq
Y\subseteq E$; it follows immediately that $N\les M$ if and only if 
$M^*\les N^*$ and that $N\les M$ implies $N\leq M$.

\subsection{Submodular functions}

Recall that a \emph{submodular function} on a set $E$ is a function
$f\colon 2^E\rta\bZ$ that satisfies $f(X)+f(Y)\geq f(X\cup Y) +f(X\cap
Y)$ for all $X,Y\subseteq E$; also, a submodular function $f$ is the
rank function of a matroid on $E$ if and only if $0\leq f(X)\leq |X|$,
for all $X\subseteq E$, and $f$ is order-preserving (that is,
$f(X)\leq f(Y)$ for all $X\subseteq Y\subseteq E$).  The following
result was stated without proof in \cite[Lemma 1, p.~315]{we:mt}.  We
provide a proof here for the convenience of the reader.

\begin{lem}\label{lem:submod}
  If $f$ and $g$ are submodular functions on $E$ such that $f-g$ is
  order-preserving, then $h\colon 2^E\rta\bZ$ given by
  $h(X)=\min\{f(X),\,g(X)\}$ is submodular.
\end{lem}

\begin{proof}
  Let $\cj\sub{\leq}=\{X\subseteq E \co f(X)\leq g(X)\}$ and
  $\cj\sub{\geq}=\{X\subseteq E \co f(X)\geq g(X)\}$.  Note that
  $\cj\sub\geq\cup\cj_\leq=2^E$ and that, since $f-g$ is
  order-preserving, $\cj\sub\leq$ is an ideal and $\cj\sub\geq$ is a
  filter in $2^E$.  If $X,Y\in\cj\sub\leq$, then 
  \begin{align*}
    h(X)+h(Y) \= &\, f(X)+f(Y) \\
    \,\geq\, &\, f(X\cup Y)+f(X\cap Y) \\
    \,\geq\, &\, h(X\cup Y)+h(X\cap Y),
  \end{align*}
  and similarly for $X,Y\in\cj\sub\geq$.  If $X\in\cj\sub\geq$ and
  $Y\in\cj\sub\leq$, then $X\cup Y\in\cj\sub\geq$ and $X\cap
  Y\in\cj\sub\leq$.  Using the fact that $(f-g)(Y)\geq (f-g)(X\cap
  Y)$, we thus have
  \begin{align*}
    h(X)-h(X\cup Y) \= &\, g(X)-g(X\cup Y) \\
    \,\geq\, & \, g(X\cap Y)  - g(Y) \\
    \,\geq\, & \, f(X\cap Y) - f(Y) = h(X\cap Y)-h(Y).
  \end{align*}
  Hence $h$ is submodular.
\end{proof}

We note that if, in addition, one of the functions $f$ and $g$ is the
rank function of a matroid and the other is nonnegative and
order-preserving, then $h$ is the rank function of a matroid.

\subsection{The Higgs lift}

Suppose that $N\les M$ in $\cm (E)$ and $i\geq 0$.  By Lemma
\ref{lem:submod} and the observation after it, the function
$\min\{\rank M, i+\rank N\}$ is the rank function of a matroid, which
we denote by \higgs iMN\ and which is known as the \emph{$i$th Higgs
  lift of} $N$ \emph{towards} $M$.  It is immediate that $\higgs 0MN =
N$, that $\higgs iMN=M$ for all $i\geq\r M-\r N$, and that $\higgs
iMN$ is an elementary quotient of $\higgs{i+1}MN$ for all $i\geq 0$.
We extend the definition of \higgs iMN\ to all $i\in\bZ$ by setting
$\higgs iMN = N$ for $i< 0$.  We define the family $\ideal <=\ide
<iMN\subseteq 2^E$ by
$$
\ide <iMN \= \{X\subseteq E\co\rkm X- \rkn X < i\},
$$
similarly define 
\ideal\leq, \ideal >, \ideal\geq, \ideal =, and let
$\ideal\succ\=
\{X\subseteq E\co\rkm X-\rkn X = i+1\}$.
The rank function of $L=\higgs iMN$ is thus given, for $i\geq 0$,
by
$$
 \rkl X \= 
\begin{cases}
\rkm X &
\text{if $X\in\ideal\leq$,}\\
\rkn X + i&
\text{if $X\in\ideal\geq$},
\end{cases}
$$
for all $X\subseteq E$.
Note that, since $N$ is a quotient of $M$ and thus $\rank M-\rank N$ is 
order-preserving,
\ideal <\ and \ideal\leq\ are order ideals, and \ideal>\ and
\ideal\geq\ are order filters, in the Boolean algebra $2^E$.

The following result about the dual of a Higgs lift is used
  extensively throughout this paper.
\begin{pro}\label{pro:hdual}
  If $N\les M$ and $i+j=\r M-\r N$, then $(\higgs
  iMN)^*=\higgs j{N^*}{M^*\!}$ and $\ide\leq i MN^c= \ide\geq
  j{N^*\!}{M^*}$.
\end{pro}

\begin{proof}
  The claim is trivially true unless $0\leq i\leq r(M)-r(N)$, so
  assume this inequality holds.  Let $L=\higgs iMN$ and $L'=\higgs
  j{N^*}{M^*\!}$.  For all $X\subseteq E$, we have
  \begin{align*}
    \rk{L'}X  \= & \,\min\{\rk{N^*}X,\, \rk{M^*}X +j\} \\
    \= & \, |X|-\r N -i +\min\{\rkm{E-X},\,  \rkn{E-X}+i\} \\
    \= & \, |X| -\r{L} +\rkl{E-X}\\
    \= & \, \rk{L^*}X,
  \end{align*}
  and so $L'=L^*$.  A straightforward computation gives
  $$\rk{N^*}X -\rk{M^*}X-j =
  -(\rkm{E-X} -\rkn{E-X} -i)$$ for all $X\subseteq E$.  Hence
  $X\in\ide\geq j{N^*\!}{M^*}$ if and only if $E-X\in\ide\leq i MN$,
  that is, if and only if $X\in\ide\leq i MN^c$.
\end{proof}

The same argument also gives $\ide <iMN^c= \ide >j{N^*\!}{M^*}$ and
$\ide =iMN^c= \ide =j{N^*\!}{M^*}$.

In the next section we define,
in terms of Higgs lifts,
the free splice operation on matroids.
Suitably interpreting the next result 
then gives the various equivalent formulations of this new operation.
\begin{thm}\label{thm:hcrypt}
  If $N\les M$ and $L=\higgs iMN$ with $i\geq 0$, then
\begin{enumerate}[\rm (1)\,]
\item 
$\ci (L) \= \ci (M)\cap\ideal\leq$,
\item
$\cs (L) \= \cs (N)\cap\ideal\geq$,
\item
$\cb (L) =\ci (M)\cap\cs (N)\cap\ideal =$, 
\item
$\cc (L) \= (\cc (M)\cap \ideal\leq)\cup(\ci (M)\cap\cf (N^*)^c\cap
\ideal\succ )$,
\item
$
 \cll X \= 
\begin{cases}
\clm X &
\text{if $X\in\ideal<$,}\\
\cln X &
\text{if $X\in\ideal\geq$},
\end{cases}\\
$
for all $X\subseteq E$,
\item
$\cf (L) \= (\cf (M)\cap\ideal<)\,\cup\,\cf (N)$,
\item
$\cz (L) \= (\cz (M)\cup\ideal >)\,\cap\,
(\cz (N)\cup\ideal <)$\\
$\hspace*{5.95ex}\= (\cz (M)\cap\ideal <)\,\cup\, (\cz (N)\cap\ideal >)\,\cup\,
(\cz (M)\cap\cz (N)).
$
\end{enumerate}
\end{thm}

\begin{proof}
  A set $X\subseteq E$ belongs to $\ci (L)$ if and only if
  $|X|=\min\{\rkm X,\rkn X +i\}$. Since $\rkm X\leq |X|$, this is the
  case if and only if $|X|=\rkm X\leq\rkn X+i$, that is, if and only
  if $X\in\ci (M)\cap\ideal\leq$, so (1) holds. Statement (2) is dual
  to (1), following from Proposition \ref{pro:hdual} and the fact that
  $\cs (L) = \ci (L^*)^c$, and (3) is immediate from (1) and (2).

  It is immediate from (1) that if $X\in\ideal\leq$, then $X\in\cc
  (L)$ if and only if $X\in\cc (M)$. For $X\in\ideal>$, we have
  $X\in\cc (L)$ if and only if $X-x\in\ci(M)$ and $\rkm {X-x} \leq
  \rkn {X-x}+i$ for all $x\in X$.  Since $\rkm X > \rkn X +i$, this is
  equivalent to $X\in\ci(M)$ and $\rkn X = \rkn{X-x}$ for all $x\in
  X$.  Hence (4) follows.
  
  Suppose $X\subseteq E$ and $Y=X\cup x$ for some $x\in E$.  If
  $X\in\ideal<$, then $Y\in\ideal\leq$; thus $\rkl Y-\rkl X =\rkm
  Y-\rkm X$, so $x\in\cll X$ if and only if $x\in\clm X$.  If
  $X\in\ideal\geq$, then $Y\in\ideal\geq$ since $\ideal\geq$ is an
  order filter; thus $\rkl Y-\rkl X =\rkn Y-\rkn X$, and so $x\in\cll X$
  if and only if $x\in\cln X$.  Thus (5) holds.  It is immediate from
  (5) that $\cf (L) = (\cf (M)\cap\ideal< ) \cup (\cf
  (N)\cap\ideal\geq )$.  The fact that $N$ is a quotient of $M$ gives $\cf
  (N)\subseteq\cf (M)$, so (6) follows.

  By Proposition \ref{pro:hdual}, we have $\cf (L^*)= \cf (\higgs
  j{N^*}{M^*\!})$, where $j=\r M-\r N-i$, and so $\cf (L^*)= (\cf
  (N^*)\cap\ide <j{N^*\!}{M^*})\cup\cf (M^*)= (\cf (N^*)\cap\ide
  >iMN^c)\cup\cf (M^*)$.  Hence
  \begin{align*}
    \cz (L) &\= \cf (L)\cap\cf (L^*)^c\\
    &\= \left((\cf (M)\cap\ideal <)\cup\cf (N)\right)\,\cap\,
    \left((\cf (N^*)^c\cap\ideal >)\cup\cf (M^*)^c\right)\\
    &\= (\cf (M)\cap\ideal<\cap\cf (M^*)^c)\cup (\cf (N)\cap\cf
    (N^*)^c\cap\ideal>) \cup (\cf (N)\cap\cf (M^*)^c).
  \end{align*}
  Since $N\les M$, and thus $M^*\les 
N^*$, we have $\cf (N)\subseteq\cf (M)$ and $\cf (M^*)\subseteq\cf
  (N^*)$.  Hence $\cf (N)\cap\cf (M^*)^c= \cf (M)\cap\cf
  (M^*)^c\cap\cf (N)\cap\cf (N^*)^c= \cz (M)\cap\cz (N)$, and thus
  $$
  \cz (L)\= (\cz (M)\cap\ideal <)\,\cup\, (\cz (N)\cap\ideal
  >)\,\cup\, (\cz (M)\cap\cz (N)).
  $$
\end{proof}

For $N\les M$, the interval from $N$ to $M$ in the strong order, which
we denote by $[N,M]\sub\les$, is graded of rank $\r M -\r N$.  The
following characterization of the Higgs lift was given in
\cite{keke:hfg}.

\begin{pro}\label{pro:hmax}
  For $N(E)\les M(E)$ and $0\leq i\leq\r M-\r N$, the Higgs lift
  \higgs iMN\ is the maximum element in the weak order on
  $\{L\in [N,M]\sub\les\co \r L-\r N = i\}$.
\end{pro}

\begin{proof}
  Suppose $L\in [N,M]\sub\les$ with $\r L-\r N = i$.  Let $X\subseteq
  E$.  Since $N\les L$, the function $\rank L-\rank N$ is order
  preserving, so $\rkl X -\rkn X\leq\rkl E - \rkn E = i$, that is,
  $\rkl X\leq \rkn X +i$.  Since $L\les M$, we have $L\leq M$, so
  $\rkl X\leq\rkm X$.  Hence $\rkl X\leq\min\{\rkm X,\rkn X+i\}=
  \rk{\higgs iMN} X$, so $L\leq\higgs iMN$.
\end{proof}

We next show that the Higgs lift preserves the strong and weak orders.

\begin{pro}\label{pro:horder}
  Suppose that $N(E)\les M(E)$ and $N'(E)\les M'(E)$.
  \begin{enumerate}[\rm (1)\,]
  \item If $N\leq N'$ and $M\leq M'$, then $\higgs iMN\leq\,\higgs
    i{M'}{N'\!}$.
  \item If $N\les N'$ and $M\les M'$, then $\higgs iMN\les\,\higgs
    k{M'}{N'\!}$ where $k=i+\r{M'}-\r M$.
  \end{enumerate}
\end{pro}

\begin{proof}
  Suppose $N\leq N'$ and $M\leq M'$.  By part (1) of Theorem
  \ref{thm:hcrypt}, it suffices to show that $\ci (M)\cap\ide\leq
  iMN\subseteq \ide \leq i{M'\!}{N'}$ in order to prove that $\higgs
  iMN\leq\higgs i{M'}{N'\!}$.  
To see the inclusion, note that if $X\in\ci (M)\cap\ide\leq iMN$,
  then $|X|=\rkm X=\rk{M'}X$ and $i\geq |X|-\rkn X \geq |X|-\rk{N'}X$,
  and thus $X\in\ide\leq i{M'\!}{N'}$.

  It is immediate from part (6) of Theorem \ref{thm:hcrypt} that
  $\cf (N)\subseteq\cf (N')$ implies $\cf (\higgs iMN) \subseteq\cf
  (\higgs iM{N'\!})$, that is, $N\les N'$ implies $\higgs iMN\les\higgs
  iM{N'\!}$.  Now suppose that $M\les M'$, and let $j=\r M -\r N
  -i$. Using this result, together with Proposition \ref{pro:hdual},
  applied twice, we have
  \begin{align*}
    \higgs iMN &\= (\higgs j{N^*}{M^*\!})^*\\
    &\,\les\, (\higgs j{N^*}{(M')^*\!})^*\\
    &\= \higgs k{M'}N
  \end{align*}
  where $k=\r{N^*}-\r{(M')^*}-j=i+\r{M'}-\r M$.
\end{proof}

We end this section with a result that expresses minors of Higgs lifts
as Higgs lifts of corresponding minors.  Recall that, by convention,
$\higgs iMN = N$, for $i<0$.

\begin{pro}\label{pro:hminor}
  If $N$ and $M$ are matroids on $E$ with $N\les M$, then for all
  $A\subseteq E$ and $i\in\bZ$,
  $$
  \higgs iMN|A \= \higgs i{M|A}{N|A} \sands \higgs iMN/A \= \higgs
  {i-k}{M/A}{N/A},
  $$
  where $k = \rkm{A} -\rkn{A}$.
\end{pro}

\begin{proof}
  The result obviously holds for $i<0$ so assume $i\geq 0$.  Let
  $L=\higgs iMN$. If $X\subseteq A$, then
  \begin{align*}
    \rk{L|A}X \= & \, \rkl X  \\
    \= & \, \min\{\rkm X,\, \rkn X+i\} \\
    \= & \, \min\{\rk{M|A}X,\,  \rk{N|A}X + i\} \\
    \= & \, \rk{\higgs i{M|A}{N|A}}X,
  \end{align*} and so the first equality in the theorem holds.  The
  second equality follows from the first by duality, as follows:
  using Proposition \ref{pro:hdual} twice, we have
  \begin{align*}
    \higgs iMN/ A \= & \, ((\higgs iMN)^*\backslash A)^* \\
    \= & \, (\higgs j{N^*}{M^*\!}\backslash A)^* \\
    \= & \, (\higgs j{N^*\backslash A} {M^*\backslash A})^* \\
   \= & \, \higgs \ell{(M^*\backslash A)^*}{(N^*\backslash A)^*\!} \\
   \= & \, \higgs \ell{M/A}{N/A}
  \end{align*}
  where $j=\r M -\r N - i$ and  
$$
\ell\=\r{N^*\backslash A} -
\r{M^*\backslash A} -j \=i -\rkm{A} +\rkn{A}.
$$ 
Thus the second equality in the theorem holds.
\end{proof}

\section{Splices}\label{sec:splice}

In the first subsection below we introduce 
splices of matroids and show that all
splices of $M(A)$ and $N(B)$ lie in a certain interval
$[\loo N, \ist M ]\sub\les$ in the strong order on $\mathcal{M}(A\cup B)$;
indeed, all splices have the same rank, and, if we order the matroids
in $[\loo N, \ist M]\sub\les$ of this rank by the weak order, then the
splices are a filter.  In the second subsection we show 
that this filter is nonempty and hence $M$ and $N$
have a freest splice, given by a Higgs lift of \loo N\ towards \ist M;
this is the free splice $M\fsp N$. 
This subsection also treats duals of
free splices, the effect of the free splice operation on the weak and strong
orders, and a variety of equivalent formulations of the free splice.

\subsection{Matched matroids and splices}

We say that matroids $M(A)$ and $N(B)$ are \emph{matched} or, more
precisely, that the ordered pair $(M,N)$ is matched, if $M$ and $N$
agree on $A\cap B$ in the sense that $M.(A\cap B)=N|(A\cap B)$. In
particular, any pair of matroids on disjoint sets is matched; also,
the pair $(M,N)$ is matched if and only if $(N^*\!,M^*)$ is matched.
If $M(A)$ and $N(B)$ are matched, then a matroid $L$ on $A\cup B$ is
called a \emph{splice} of $M$ and $N$ if $M=L|A$ and $N=L.B$.  We
write \splice MN\ for the set of all splices of $M$ and $N$.  Note
that, for any matroid $L(A\cup B)$, the matroids $L|A$ and $L.B$ are
matched and $L\in\splice{L|A}{L.B}$. Also note that $L\in\splice MN$
if and only if $L^*\in\splice{N^*\!}{M^*}$, and that the rank of
any $L\in\splice MN$ is equal to $\rkm{A-B}+\r N=
\r M+\r N -\rkn{A\cap B}$.
  
Figure~\ref{whirl} shows matched matroids $M$ and $N$, together with the
set $\splice MN$, ordered by the weak order.
The freest of these splices is the $3$-whirl, $\mathcal{W}^3$.
Section~\ref{sec:class} contains additional examples of splices.

\begin{figure}
\begin{center}
\includegraphics[width = 4.5truein]{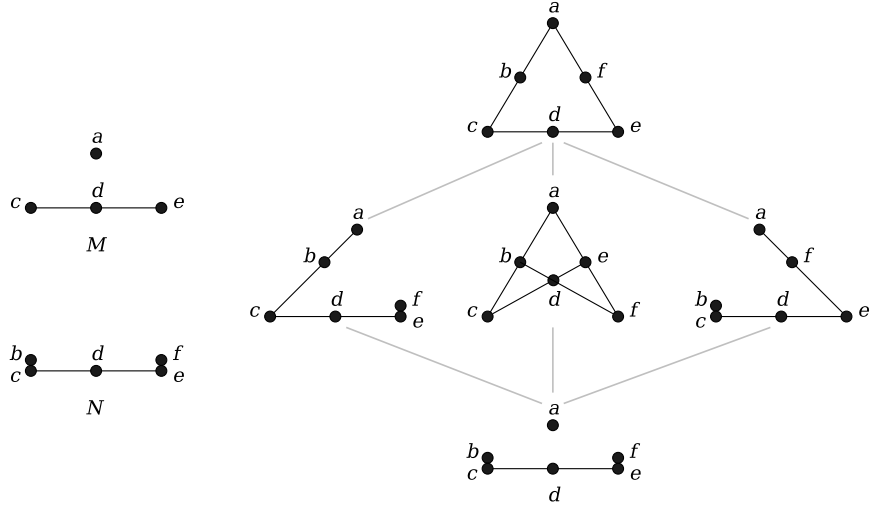}
\end{center}
\caption{Matched matroids $M$ and $N$, along with their splices; the
  grey lines indicate the weak order.}\label{whirl}
\end{figure}

If $M(A)$ and $N(B)$ are matched, we denote by $\ist M$ the matroid $M
\oplus I(B-A)$, obtained by adjoining the elements of
$B-A$ to $M$ as isthmuses, and by $\loo N$ the matroid $N\oplus
I^*(A-B)$, obtained by adjoining the elements of $A-B$ to $N$ as loops.
Note that $\rk{\ist M}X =\rkm{X\cap A} +|X-A|$ and $\rk{\loo N}X =
\rkn{X\cap B}$ for all $X\subseteq A\cup B$.

\begin{pro}\label{pro:quot}
  If $M(A)$ and $N(B)$ are matched, then $\loo N\les\ist M$.
\end{pro}

\begin{proof}
  If $F\in\cf (N)$, then $F\cap A\in\cf (N|(A\cap B))=\cf (M.(A\cap
  B))$; therefore $(A-B)\cup (F\cap A)\in\cf (M)$.  The inclusion
  $B-A\subseteq\isthm{\ist M}$ now gives $(A-B)\cup F \in\cf(\ist M)$
  for any $F\in\cf (N)$.  Since any flat of \loo N\ has the form
  $(A-B)\cup F$ for some $F\in\cf (N)$, we get $\cf (\loo
  N)\subseteq\cf (\ist M)$.
\end{proof}
 The next result identifies some order-theoretic structure of the set
  of splices of two matched matroids.
\begin{pro}\label{pro:spfilt}
  If $M(A)$ and $N(B)$ are matched, then \splice MN\ is a filter in
  the weak order on the set $\mathcal{L}=\{L\in [\loo N, \ist
  M]\sub\les\co \r L-\r N =\rkm{A-B}\}$.
\end{pro}

\begin{proof}
  Let $L\in\splice MN$.  Since $L|A=M$, we have $\cf (M) =\{X\cap A\co
  X\in\cf (L)\}$ and hence $\cf (L)\subseteq \cf (M_1)=\{X\subseteq
  A\cup B\co X\cap A\in\cf (M)\}$, that is, $L\les\ist M$.  Now
  $L\in\splice MN$ gives $L^*\in\splice{N^*}{M^*}$, hence
  $L^*\les\ist{(N^*)}=(\loo N)^*$ or, equivalently, $\loo N\les L$.
  It is clear that $\r L-\r{\loo N}=\rkm{A-B}$; hence $\splice
  MN\subseteq\mathcal L$.

  Suppose that $L\in\splice MN$ and that $P\in\mathcal L$ satisfies
  $L\leq P$.  Restriction preserves both the weak and strong orders;
  hence $M=L|A\leq P|A\les\ist M|A=M$, so $P|A=M$.  Since duality
  reverses strong order, and preserves weak order for matroids of the
  same rank, we have $L^*\leq P^*\les (\loo N)^*=\ist{(N^*)}$ and so the
  same argument gives $P^*|B=N^*$\!, that is, $P.B=N$.  Hence
  $P\in\splice MN$.
\end{proof}

\subsection{The free splice}\label{subsec:free}

Given matched matroids $M(A)$ and $N(B)$, we denote by $M\fsp N$ the
Higgs lift \higgs i{\ist M}{\loo N}, where $i=\rkm{A-B}$.  We will
show that $M\fsp N$ is the freest splice of $M$ and $N$.  Denote by
$\ideal<$, or $\idea <MN$, the order ideal $\ide <i{\ist M}{\loo N}$, and
similarly define \ideal\leq, \ideal >, \ideal\geq, and \ideal =.
Hence $$\ideal< =\{X\subseteq A\cup B\co \rk{\ist M}X-\rk{\loo
  N}X<\rkm{A-B}\},$$ that is,
$$
\ideal <\=\{X\subseteq A\cup B\co \rkm{X\cap A}+|X-A| < \rkn{X\cap B}
+\rkm{A-B}\}.
$$
The rank function of $M\fsp N$ is thus given by
\begin{align}\label{eq:rank}
  \rk{M\fsp N}X & \= \min\{\rkm{X\cap A} + |X-A|,\,\rkn{X\cap B}+\rkm{A-B}\}\\
&\=
\begin{cases}
\rkm{X\cap A} + |X-A| &
\text{if $X\in\ideal\leq$,}\\
\rkn{X\cap B}+\rkm{A-B}&
\text{if $X\in\ideal\geq$},\notag
\end{cases}
\end{align}
for all $X\subseteq A\cup B$.

Note that $\idea \leq {N^*}{M^*}=\ide \leq j{\ist {(N^*)}}{\loo
    {(M^*)}}$, where $j=\rk{N^*}{B-A}$. It follows that $i+j =
  \r{\ist M}-\r{\loo N}$, and so by Proposition \ref{pro:hdual}
we have the following lemma.
\begin{lem}\label{lem:little}
 If $M$ and $N$ are matched, then $\idea \leq {N^*\!}{M^*}=
\idea\geq MN^c$.
\end{lem}

Given sets $A\subseteq B\subseteq E$, we use the standard notation $[A,B]$ for
the interval $\{X\co A\subseteq X\subseteq B\}$ in the Boolean algebra $2^E$.
\begin{lem}\label{lem:intideal}
  If $M(A)$ and $N(B)$ are matched, then
  \begin{enumerate}[\rm (1)\,]
  \item 
    $[\emptyset, A]\subseteq\ideal\leq$; also $X\in [\emptyset , A]$
    belongs to \ideal =\ if and only if $A-B\subseteq \clm X$,
    \item 
    $[A-B, A\cup B]\subseteq\ideal\geq$; also $X\in [A-B, A\cup B]$
    belongs to \ideal =\ if and only if $B-A\subseteq\cl{N^*}{B-X}$,
    that is, if and only if $X-A\subseteq\isthm{N|(X\cap B)}$,
  \item 
    $[A-B,A]\subseteq\ideal =$.
  \end{enumerate}
\end{lem}

\begin{proof}
  If $X\subseteq A$, then $\rk{\ist M}X =\rkm X$ and
  \begin{align*}
    \rk{\loo N}X \= & \, \rkn{X\cap B} \\
    \= & \, \rk{N|(A\cap B)}{X\cap B} \\
   \= & \, \rk{M.(A\cap B)} {X\cap B} \\
   \= & \, \rkm{X\cup (A-B)}-\rkm{A-B}.
  \end{align*}
  Since $\rkm X\leq\rkm{X\cup (A-B)}$, we thus have $\rk{\ist
    M}X\leq\rk{\loo N}X+\rkm{A-B}$, with equality if and only if
  $A-B\subseteq\clm X$.  Hence (1) holds.  The second statement is the
  dual of the first: if $A-B\subseteq X$, then $(A\cup B)-X\subseteq B$ and
  so, by part (1) together with Lemma~\ref{lem:little}, we have
  $(A\cup B)-X\in\idea\leq{N^*\!}{M^*}= \idea\geq MN^c$, that is,
  $X\in\ideal\geq$, with equality if and only if
  $B-A\subseteq\cl{N^*}{B-X}$.  The third statement is an obvious
  consequence of the first two.
\end{proof}

\begin{cor}\label{cor:intrank}
  If $M(A)$ and $N(B)$ are matched, then a matroid $L(A\cup B)$
  belongs to \splice MN\ if and only if $\rkl X =\rk{M\fsp N}X$ for
  all $X\in [\emptyset, A]\cup [A-B, A\cup B]$.
\end{cor}

\begin{proof}
  Lemma \ref{lem:intideal} and Equation \eqref{eq:rank} give $\rk{M\fsp
    N}X= \rkm X$ for all $X\subseteq A$ and $\rk{M\fsp N}X=\rkn{X\cap
    B}+\rkm{A-B}$ for all $X\supseteq A-B$.  However, $L\in\splice MN$
  if and only if $L|A=M$ and $L.B=N$, that is, if and only if $\rkl X
  = \rkm X$ for all $X\subseteq A$ and $\rkl X =\rkn{X\cap
    B}+\rkm{A-B}$ for all $X\supseteq{A-B}$.
\end{proof}

\begin{thm}\label{thm:freest} 
  If $M$ and $N$ are matched, then $M\fsp N$ is the freest splice of
  $M$ and $N$, that is, $M\fsp N$ is the maximum element of \splice
  MN\ in weak order.
\end{thm}

\begin{proof}
  It is immediate from Corollary \ref{cor:intrank} that $M\fsp
  N\in\splice MN$. Hence the result follows directly from Propositions
  \ref{pro:hmax} and \ref{pro:spfilt}.
\end{proof}

We refer to $M\fsp N$ as the \emph{free splice} of $M$ and $N$.  Note
that if $A\subseteq B$, and so $M=N|A$, then $M\fsp N=N$ is the unique
splice of $M$ and $N$. Similarly, if $B\subseteq A$ then $M\fsp N=M$
is the unique splice of $M$ and $N$.  If $A$ and $B$ are disjoint,
then the free splice $M\fsp N$ is the \emph{free product}
$M\frp N$ of $M$ and $N$, which was introduced in \cite{crsc:fpm}.  In
this case, the set \splice MN\ has minimum element $M\oplus N$.  The
following example shows that, in general, \splice MN\ does not have a
unique minimal element.

\begin{exa}\label{exa:mins}
  Let $M$ be the uniform matroid $U_{2,3}$ on $\{a,b,c\}$.  Let $N$ be
  the rank-$2$ matroid on $\{b,c,d,e\}$ in which $\{b,c\}$ is the only
  nonspanning circuit.  The free splice of $M$ and $N$ is a plane on
  $\{a,b,c,d,e\}$ in which $\{a,b,c\}$ is the only nonspanning
  circuit. There are only two other splices of $M$ and $N$; one has
  $b,d,e$ collinear and the other has $c,d,e$ collinear.  Thus,
\splice MN\ has more than one minimal element. 
\end{exa}

The next result gives the simple relation between the free splice and
duality.

\begin{pro}\label{pro:dual}
  If $M(A)$ and $N(B)$ are matched, then $(M\fsp N)^* = N^*\fsp M^*$.
\end{pro}

\begin{proof}
 The result is immediate  from Proposition \ref{pro:hdual} and
Lemma \ref{lem:little}.
Alternatively, it follows from
  Theorem \ref{thm:freest} and the fact that the duality operator
  is a poset isomorphism $\splice MN\rta\splice{N^*\!}{M^*}$.
\end{proof}

The following proposition shows that the free splice operation
preserves the weak order and that the operator defined by taking the
free splice with a fixed matroid, on either the left or the right,
preserves the strong order.

\begin{pro}\label{pro:splorder}
  Suppose that $M(A)$ and $N(B)$ are matched.
  \begin{enumerate}[\rm (1)\,]
  \item If $(M'\!, N')$ is matched, with $M\leq M'$ and $N\leq N'$, then
    $M\fsp N\leq M'\fsp N'$.
  \item If $(M, N')$ is matched and $N\les N'$, then $M\fsp N\les
    M\fsp N'$.
  \item If $(M'\!, N)$ is matched and $M\les M'$, then $M\fsp N\les
    M'\fsp N$.
  \end{enumerate}
\end{pro}

\begin{proof}
  Suppose $(M'\!,N')$ is matched. If $M\leq M'$ and $N\leq N'$, then
  $\ist M\leq\ist{(M')}$ and $\loo N\leq\loo{(N')}$; also,
  $\rkm{A-B}\leq\rk{M'}{A-B}$.  Thus, statement (1) follows
  immediately from part (1) of Proposition \ref{pro:horder}.

  If $(M,N)$ and $(M, N')$ are matched and $N\les N'$, then $\loo
  N\les\loo{(N')}$, so statement (2) follows from part (2) of
  Proposition \ref{pro:horder}.  Now if $(M', N)$ is matched and
  $M\les M'$, then $(M')^*\les M^*$, so by part (2) and Proposition
  \ref{pro:dual} we have
  \begin{align*}
    (M'\fsp N)^* & \= N^*\fsp (M')^*\\
    &\,\les\, N^*\fsp M^*\\
    &\= (M\fsp N)^*,
  \end{align*}
  and hence $M\fsp N\les M'\fsp N$.
\end{proof}

The equivalent formulations of the Higgs lift in Theorem
\ref{thm:hcrypt} yield
the following equivalent formulations of the
free splice.
\begin{thm}\label{thm:crypt}
  Suppose that $M(A)$ and $N(B)$ are matched. Let $L=M\fsp N$, $E=A\cup
  B$, $S=A-B$, and $r=\r L=\rkm S +\r N$.  \medskip
  \begin{enumerate}[\rm (1)\,]
  \item %
    $\ci (L) \= \{ X\subseteq E\co\text{$X\cap A\in\ci (M)$ and\/
      $|X|\leq\rkm S +\rkn{X\cap B}$}\}$.\\
  \item %
    $\cs (L)\= \{ X\subseteq E\co\text{$X\cap B\in\cs (N)$ and\/
      $\rkm{X\cap A} +|X-A|\geq r$}\}$.\\
  \item %
    $\cb (L)\=\{X\subseteq E\co
    \text{$|X|=r$, $X\cap A\in\ci (M)$, and\/ $X\cap B\in\cs (N)$}\}.$\\
  \item %
$\cc (L)\=\cc (M) \cup\cc$, where $\cc$ is the set of all
  $X\subseteq E$ such that $X\cap A\in\ci (M)$ and\/ $\isthm{N|(X\cap
    B)}=\emptyset$, with $\rkn{X\cap B} = |X|-\rkm S -1$.\\ 
\item %
    $\cll X\=
    \begin{cases}
      \clm{X\cap A}\cup X &
      \text{if $\rkm{X\cap A} +|X-A|<\rkm S +\rkn{X\cap B}$},\\
      \cln{X\cap B}\cup S
      &\text{otherwise},\\
    \end{cases}
    $

    \medskip
    \noindent
    for all $X\subseteq E$.\\
  \item %
    $\cf (L)= \cf_1\cup\cf_2$, where $\cf_2= \{S\cup F\co F\in\cf
    (N)\}$ and $\cf_1$ is the family $\{X\subseteq E\co\text{$X\cap
      A\in\cf (M)$ and\/
      $\rkm{X\cap A}+|X-A|<\rkm S +\rkn{X\cap B}$}\}$.\\
  \item %
    $\cz (L) = \cz_1\cup\cz_2\cup \cz_3$, where
    \begin{align*}
      \cz_1&\=
      \{ Z\in\cz (M)\co S\nsubseteq Z\},\\
      \cz_2 &\= \{ Z\in\cz (M)\co
      \text{$S\subseteq Z$ and $Z\cap B\in\cz (N)$}\},\\
      \cz_3 &\= \{ Z\subseteq E\co\text{$S\subseteq Z\nsubseteq A$ and
        $Z\cap B\in\cz (N)$}\}.
    \end{align*}
  \end{enumerate}
\end{thm}

\begin{proof}
  Each statement of the theorem follows from its counterpart in
  Theorem \ref{thm:hcrypt}, with \ist M\ and \loo N\ replacing $M$
  and $N$, together with the observations below.
  \begin{enumerate}[(1)\,]
  \item %
    $\ci (\ist M) =\{X\subseteq E\co X\cap A\in\ci (M)\}$; also,
    for $X\in\ci (\ist M)$, we have
    $X\in\ideal\leq$ if and only if $|X|\leq\rkn{X\cap B}+\rkm
    S$.  \smallskip
  \item %
    $\cs (\loo N ) = \{X\subseteq E\co X\cap B\in\cs (N)\}$, and if
    $X\cap B\in\cs (N)$, then $\rkn{X\cap B}+\rkm S= r$.  \smallskip
  \item %
    The statement is immediate from the previous two.  \smallskip
  \item %
    Note that $\cc (\ist M)=\cc (M)$ and $\cc (M)\subseteq \ideal\leq$
    by part (1) of Lemma~\ref{lem:intideal}.  Also, $E-X\in \cf ((\loo
    N)^*)$ if and only if $X$ is a union of circuits of $\loo N$, that
    is, if and only if $\isthm{N|(X\cap B)}=\emptyset$.  If $X\in\ci (\ist M)$, that
    is, if $X\cap A\in\ci (M)$, then $X\in\ideal\succ$ if and only if
    $|X| -1 =\rkn{X\cap B}+\rkm S$.  \smallskip
  \item %
    $\cl{\ist M}X = \clm{X\cap A}\cup X$ and $\cl{\loo N}X= \cln{X\cap
      B}\cup S$ for all $X\subseteq E$.  \smallskip
  \item %
    $\cf (\ist M)=\{X\subseteq E\co X\cap A\in\cf (M)\}$ and $\cf
    (\loo N)=\cf_2$.  \smallskip
  \item %
 $\cz (\ist M) =\cz (M)$ and $\cz (\loo N)=\{Z\cup S\co Z\in\cz (N)\}$,
so $\cz (\ist M)\cap
  \cz (\loo N)=\cz_2$.  By Lemma \ref{lem:intideal}, we have $\cz
  (\ist M)\subseteq\ideal\leq$ and $\cz (\loo N)\subseteq\ideal\geq$.
  Lemma \ref{lem:intideal} also tells us that if $X\subseteq A$ is
closed in $M$, then $X\in\ideal<$ if and only if $S\nsubseteq X$,
  and that if $X\supseteq S$ is such that $X\cap B$ is cyclic in $N$,
  then $X\in\ideal>$ if and only if $X\nsubseteq A$.  Hence $\cz (\ist
  M)\cap\ideal<=\cz_1$ and $\cz (\loo N)\cap\ideal> = \cz_3$.
  \end{enumerate}
\end{proof}
 The following characterization of loops in free splices is immediate
  from part (5) of Theorem \ref{thm:crypt}, or directly from Equation
  \eqref{eq:rank}; duality gives the characterization of isthmuses.
\begin{cor}\label{cor:loopisth}
  If $M(A)$ and $N(B)$ are matched, then
\begin{align*}
\loops{M\fsp N}& \=
\begin{cases}
  \loops M \cup \loops N &
\text{if $A-B\subseteq \loops M$,}\\
\loops M &
\text{otherwise}
\end{cases}
\intertext{and}
\isthm{M\fsp N}& \=
\begin{cases}
  \isthm M \cup \isthm N &
\text{if $B-A\subseteq \isthm N$,}\\
\isthm N &
\text{otherwise}.
\end{cases}
\end{align*}
\end{cor}

\section{Factorization with respect to free splice}\label{sec:factor}

If a matroid $L(E)$ factors with respect to the free splice operation,
the factorization is necessarily of the form $L=L|A\fsp L.B$ for some
ordered pair $(A,B)$ of subsets of $E$ with $A\cup B=E$. In this case
we call the pair $(A,B)$ a \emph{free separator} of $L$.  A free
separator $(A,B)$ is \emph{nontrivial} if both $A-B$ and $B-A$ are
nonempty.  A matroid is \emph{reducible} if it has a nontrivial free
separator; otherwise it is \emph{irreducible}.  In this section we
present a number of results on the structure of free separators.  In
particular, Theorem \ref{thm:freesep} characterizes free separators of
matroids in terms of their cyclic flats.  This result, which will be
used extensively throughout the rest of
  this paper, immediately gives a characterization of irreducible
  matroids.

We write $\fs L$ for the set of all free separators of $L$.  The set
$\fs L$ is partially ordered by setting $(A,B)\leq (A',B')$ if and
only if $A\subseteq A'$ and $B\subseteq B'$.  A free separator of $L$
is \emph{minimal} if it is minimal with respect to this ordering.
Note that the pair $(A, E-A)$ is a free separator of $L$ if and only
if $A$ is a free separator of $L$ in the sense of \cite{crsc:uft},
that is, if and only if $L$ is the free product $L|A\frp L/A$.  Note
that $(A,B)$ is a (minimal) free separator of $L$ if and only if
$(B,A)$ is a (minimal) free separator of $L^*$.

For example, the $3$-whirl $\mathcal{W}^3$ (the maximum element of
\splice MN\ in Figure~\ref{whirl})
has among its free separators
$(\{a,c,d,e\},\{b,c,d,e,f\})$ and $(\{a,c,d,e,f\},\{b,c,d,e,f\})$; the
first of these is minimal.

 The next lemma is used in proving the main result of this
  section, Theorem \ref{thm:freesep}.
\begin{lem}\label{lem:wkcyc}
  Suppose $L\leq P$ in the weak order.  If $\rkl Z=\rkp Z$ for
  all $Z\in\cz (L)$, then $L=P$.
\end{lem}

\begin{proof}
  If $L < P$, then there is some $C\in\cc (L)$ with $C\in\ci (P)$; however
  $Z=\cll C\in\cz (L)$ and $\rkl Z = |C|-1<|C|=\rkp C\leq\rkp Z$.
\end{proof}

\begin{thm}\label{thm:freesep}
  A pair of sets $(A,B)$ is a free separator of a matroid $L(A\cup B)$ if
  and only if $\cz (L)\subseteq [\emptyset, A]\cup
  [A-B, A\cup B]$.
\end{thm}

\begin{proof}

Theorem \ref{thm:freest} implies that $L\leq L|A\fsp L.B$ for 
any matroid $L(A\cup B)$. If $\cz (L)\subseteq [\emptyset ,A]\cup
  [A-B, A\cup B]$, then it follows from Corollary \ref{cor:intrank} that
  $\rkl Z=\rk{L|A\fsp L.B}Z$ for all $Z\in\cz (L)$.  Hence Lemma
  \ref{lem:wkcyc} implies that $L=L|A\fsp L.B$, that is, $(A,B)$ is a
  free separator of $L$.  On the other hand, if $(A,B)$ is a free
  separator of $L$, then $\cz (L)\subseteq [\emptyset, A]\cup
  [A-B, A\cup B]$, by Theorem \ref{thm:crypt}.
\end{proof}

The following corollary is immediate.

\begin{cor}\label{cor:fsfilt}
  If $(A,B)$ is a free separator of $L(E)$, and $A', B'\subseteq E$
  satisfy $A\subseteq A'$ and $B\subseteq B'$, then $(A',B')$ is a
  free separator of $L$; in other words \fs L\ is an order filter in
  the Boolean algebra $2^E\times 2^E$.
\end{cor}
The following result  provides a much
more efficient test for irreducibility than direct application of
Theorem \ref{thm:freesep}.
\begin{cor}\label{cor:bfsep}
  A matroid $L(E)$ is irreducible if and only if, for each ordered
  pair $x,y$ of distinct elements of $E$, there is some $Z\in\cz (L)$
  with $x\in Z$ and $y\notin Z$.
\end{cor}

\begin{proof}
  If there is a pair $x,y\in E$ such that every $Z\in\cz(L)$
  containing $x$ also contains $y$, then, by Theorem
  \ref{thm:freesep}, the pair $(E-x, E-y)$ is a free separator of $L$.
  On the other hand, if $(A,B)$ is a free separator of $L$, and $x\in
  B-A$ and $y\in A-B$, then, again by Theorem \ref{thm:freesep}, any
  cyclic flat of $L$ that contains $x$ also contains $y$.
\end{proof}
We note the following immediate consequence of Corollary \ref{cor:bfsep}.
\begin{cor}
 If $|E|\geq 2$ and $L(E)$ contains either a loop or an isthmus, then
  $L$ is reducible.
\end{cor}

Recall that elements $x$ and $y$ of a matroid $L$ are \emph{clones} if
the map switching $x$ and $y$ and fixing all other elements is an
automorphism of $M$.  Equivalently, $x$ and $y$ are clones if they are
contained in precisely the same cyclic flats.  The following corollary
is immediate from Corollary \ref{cor:bfsep}.

\begin{cor}
  Any matroid having clones is reducible.  Any identically self-dual
  matroid without clones is irreducible.
\end{cor}

\begin{cor}\label{cor:clones}
  For a matroid $L(E)$ and proper subsets $A$ and $B$ of $E$ with
  $A\cup B=E$, both $(A,B)$ and $(B,A)$ are free separators of $L$ if
  and only if all elements of $(A-B)\cup (B-A)$ are clones in $L$.
\end{cor}
 We next address the question of
when a pair $(A',B')\in 2^E\times 2^E$ below a free
  separator $(A,B)$ of $L$ is also a free separator of $L$.
\begin{pro}\label{pro:red}
  Suppose $L(E)=M(A)\fsp N(B)$.  If $A'\subseteq A$ and $B'\subseteq
  B$ satisfy $A'\cup B'=E$, then $(A',B')\in\fs L$ if and only if
  $(A',A\cap B')\in\fs M$ and $(A'\cap B,B')\in\fs N$.
\end{pro}

\begin{proof}
  If $(A',B')\in\fs L$, then $\cz (L)\subseteq [\emptyset ,A']\cup
  [A'-B', E]$, by Theorem \ref{thm:crypt}.  Since any cyclic flat of
  $M=L|A$ is of the form $Z\cap A$, for some $Z\in\cz (L)$, it follows
  that $\cz (M)\subseteq [\emptyset, A']\cup [A'-B', A]=[\emptyset,
  A']\cup [A'-(A\cap B'), A]$ so, by Theorem \ref{thm:freesep}, we
  have $(A',A\cap B')\in\fs M$. Duality then gives $(A'\cap B,B')\in\fs N$.

  Now suppose $(A',A\cap B')\in\fs M$ and $(A'\cap B,B')\in\fs N$.  We
  have (i) $\cz (M)\subseteq [\emptyset,A']\cup [A'-B', A]$ and (ii)
  $\cz (N)\subseteq [\emptyset,A'\cap B]\cup [B-B', B]$ by Theorem
  \ref{thm:crypt} and simple manipulation.  We must show that each
  $Z\in\cz (L)$ is in $[\emptyset,A']\cup [A'-B', E]$.  This holds by
  inclusion (i) if $Z\in\cz(M)$, which is the case if $Z\in
  \cz_1\cup\cz_2$ (in the notation of Theorem~\ref{thm:crypt}, part
  (7)).  Otherwise, $Z\in \cz_3$, and so $A-B\subset Z$; also, since
  $Z\cap B\in\cz(N)$, inclusion (ii) gives $B-B'\subseteq Z\cap B$,
  and so, as needed, $A'-B'\subseteq Z$.
\end{proof}

 In the next corollary, part (1) follows from
  Proposition~\ref{pro:red}, part (2) follows from Theorem
  \ref{thm:freesep}, and part (3) follows from part (2).
\begin{cor}\mbox{}
  \begin{enumerate}[\rm (1)\,]
  \item If $(A,B)$ is a free separator of $L$, and $L|A$ and $L.B$ are
    irreducible, then $(A,B)$ is minimal in $\fs L$.
  \item Suppose $(A,B)\in\fs L$ and $x\in A\cap B$.
    \begin{enumerate}
    \item If $x\in\cll{A-B}$, then $(A, B-x)\in\fs L$.
    \item If $x\in\cl{L^*}{B-A}$, then $(A-x, B)\in\fs L$.
    \end{enumerate}
  \item If $(A,B)$ is a minimal free separator of $L$, and $M=L|A$ and
    $N= L.B$, then $A-B\in\cf (M)$ and $\isthm{N|(A\cap
      B)}=\emptyset$.
     \end{enumerate}
\end{cor}

We next note that, for any pair $(A,B)$ of subsets of $E$
with $A\cup B=E$, the free splice operation determines
a closure operator on the set of matroids $\cm (E)$,
ordered by the weak order.

\begin{pro}
  For $A,B\subseteq E$ with $A\cup B=E$, the map $\varphi\subb AB\co
  \cm (E)\rta\cm (E)$ given by $\varphi\subb AB(L)=L|A\fsp L.B$ is a
  closure operator.  A matroid $M$ is $\varphi\subb AB$-closed if and
  only if $(A,B)$ is a free separator of $M$.  If $A'\subseteq A$,
  $B'\subseteq B$, and $A'\cup B'=E$, then
  $\varphi\subb{A'}{B'}(M)\geq\varphi\subb AB(M)$ for all $M\in\cm
  (E)$.
\end{pro}

Hence the correspondence $(A,B)\mapsto \varphi\subb AB$ is an
order-reversing map from the order filter $\{ (A,B)\co A\cup B=E\}$ in
$2^E\times 2^E$ to the set of closure operators on $\cm (E)$, with the
pointwise order.

We end this section with a result showing that, for matched matroids
$M(A)$ and $N(B)$, the Higgs lift \higgs i{\ist M}{\loo N}\ is in fact
a free splice for all $i$, not just for $i=\rkm{A-B}$.  First we
recall that, given a matroid $P(E)$, an integer $i$, and $A\subseteq
E$, the $i$-\emph{fold principal lift} $\rlift iAP$ and
\emph{principal truncation} $\rtrun iAP$ \emph{of} $P$ \emph{relative to} $A$
are the matroids on $E$ defined by
$$
\rlift iA P \= \higgs i{P\backslash A\oplus I(A)}P
\sands
\rtrun iA P \= \higgs jP{P/A\oplus I^*(A)},
$$
where $j=\r P -\r{P/A\oplus I^*(A)}-i=\rkp A-i$.  Note that
$(\rlift iAP )^* =\rtrun iA{(P^*)}$, by
Proposition \ref{pro:hdual}.

\begin{pro}\label{pro:hsplice}
  If $M(A)$ and $N(B)$ are matched, and $i=\rkm{A-B}$, then
  $$
  \higgs j{\ist M}{\loo N}\=
  (\rtrun{i-j}{A-B}M)\fsp\,\rlift{j-i}{B-A}N \=
  \begin{cases}
    M\fsp\,\rlift{j-i}{B-A}N &
    \text{if $j\geq i$,}\\
    (\rtrun{i-j}{A-B}M)\fsp N& \text{if $j\leq i$},
  \end{cases}
  $$
  for all $j$.
\end{pro}

\begin{proof}
  Since $\cz (\ist M)\subseteq [\emptyset, A]$ and $\cz (\loo
  N)\subseteq [A-B,E]$, it follows from Theorem \ref{thm:hcrypt}
  that $\cz (\higgs j{\ist M}{\loo N})\subseteq [\emptyset, A]\cup
  [A-B,E]$, and hence by Theorem \ref{thm:freesep} that $\higgs j{\ist
    M}{\loo N} = \higgs j{\ist M}{\loo N} |A \fsp \higgs j{\ist
    M}{\loo N}.B$.  By Proposition \ref{pro:hminor}, we have $\higgs
  j{\ist M}{\loo N}|A = \higgs j{\ist M|A}{\loo N|A}$ and $\higgs
  j{\ist M}{\loo N}.B = \higgs{j-k}{\ist M.B}{\loo N.B}$, where
  $k=\rk{\ist M}{A-B}-\rk{\loo N}{A-B}=\rkm{A-B}=i$. It is clear that
  $\ist M|A=M$ and $\loo N.B = N$ and, since $M$ and $N$ are matched,
  we have $\ist M.B=N|(A\cap B)\oplus I(B-A)$ and $\loo N|A= M.(A\cap
  B)\oplus I^*(A-B)$; hence the result follows.
\end{proof}

\section{Interaction between free splice and other constructions}
\label{sec:inter}

Among the results proven in this section are that direct sums and
generalized parallel connections of irreducible matroids are
irreducible.  This section also shows how to express minors of free
splices as free splices.  Also, while the free product is a special
case of the free splice,
  we show how to obtain the free splice as an intersection of certain
  free products.

\subsection{Direct sum}

\begin{pro}\label{pro:sumform}
  If $M(A)$, $N(B)$, $P(C)$ are matroids with $(M,N)$ matched and
  $(A\cup B)\cap C= \emptyset$, then $(M\fsp N)\oplus P = (M\oplus
  P)\fsp (N\oplus P)$.
\end{pro}

\begin{proof}
  Let $L=(M\fsp N)\oplus P$ and $E=A\cup B\cup C$.  Since $\cz (L)$ is
  given by $\{X\cup Y\co\text{$X\in\cz (M\fsp N)$ and $Y\in\cz
    (P)$\}}$ and $\cz (M\fsp N)\subseteq [\emptyset, A]\cup [A-B,
  A\cup B]$, we have $\cz (L)\subseteq [\emptyset , A\cup C]\cup [A-B,
  E]=[\emptyset, A\cup C] \cup [(A\cup C)-(B\cup C), E]$, and hence
  $(A\cup C, B\cup C)\in\fs L$.  Therefore $L=L|(A\cup C)\fsp L.(B\cup
  C)$, that is, $(M\fsp N)\oplus P = (M\oplus P)\fsp (N\oplus P)$.
\end{proof}

\begin{pro}\label{pro:sum}
  If $L=M(A)\oplus N(B)$ and $|A|,|B|\geq 2$, then $L$ is irreducible
  if and only if $M$ and $N$ are irreducible.
\end{pro}

\begin{proof}
  Suppose $M$ and $N$ are irreducible.  Let $x$ and $y$ be distinct
  elements of $A\cup B$. If
  $x,y\in A$, then some $Z\in\cz (M)\subseteq\cz (L)$ contains $x$ and
  not $y$; a similar argument applies if $x,y\in B$.  Suppose $x\in A$
  and $y\in B$.  Since $M$ is irreducible and $|A|\geq 2$, it follows
  that $M$ has no isthmuses, so $A\in\cz (M)$.  Hence $A$ is a cyclic
  flat of $L$ that contains $x$ and not $y$.  A similar argument
  applies if $x\in B$ and $y\in A$, so $L$ is irreducible by Corollary
  \ref{cor:bfsep}.  The converse is immediate from Proposition
  \ref{pro:sumform}.
\end{proof}

\subsection{Generalized parallel connection}

We start by recalling the generalized parallel connection of matroids
$M(A)$ and $N(B)$. 
(For more information on this operation,
see~\cite{br:mcc,ox:mt}.)  
 Set $T=A\cap B$ and
$K=M|T$.  If $N|T=K$, if $\clm T$ is a modular flat of $M$,
and if each element of $\clm T-T$ is either a loop or parallel to
an element of $T$, then there is a freest amalgam of $M$ and $N$.
This freest amalgam, denoted $P_K(M,N)$, is the \emph{generalized
  parallel connection} of $M$ and $N$; its flats are the subsets
$F$ of $A\cup B$ for which $F\cap A$ is a flat of $M$ and
$F\cap B$ is a flat of $N$.  Thus, in particular, $P_K(M,N)=M\oplus N$ if
$T=\emptyset$.   Note that the closure in $P_K(M,N)$ of $X\subseteq A$ is
  $\clm{X}\cup \cln{\clm{X} \cap T}$; we will denote this closure by
  $\cl{}X$.  A similar formula for $\cl{}X$ holds if $X\subseteq B$.
\begin{pro}\label{pro:gpc}
  Let $M$ and $N$ be as above, each with at least two elements.
  If $M$ and $N$ are both irreducible, then so is $P_K(M,N)$.
\end{pro}

\begin{proof}
Since $|A|,|B|\geq 2$, the matroids $M$ and $N$, being irreducible,
  have no loops, no isthmuses, and no parallel elements.
In particular, $\clm{T}=T$.  It suffices to show that
  if $x$ and $y$ are distinct elements of $A\cup B$, then some
  cyclic flat of $P_K(M,N)$ contains $x$ and not $y$.  This
  property follows easily if $x$ and $y$ are both in $A$ or both in
  $B$ since $M$ and $N$ are restrictions of $P_K(M,N)$.

  Assume $x\in B-T$ and $y\in A-T$.  Since $N$ has no isthmuses, $x$ is
in some circuit $C$ of $N$.  No element of $A-T$ is in $\cl{}C =
\cln{C}\cup \clm{\cln{C} \cap T}$, so $\cl{}C$ is the required cyclic flat.

  Now assume $x\in A-T$ and $y\in B-T$.  Since $M$ has no
  isthmuses, $x$ is in a circuit of $M$.  If $C\in \mathcal{C}(M)$
  with $x\in C$ and $\rkm{\clm{C}\cap T}\leq 1$, then the
  description of $\cl{}C$ 
  preceding this proposition and the fact that $N$ has neither
  loops nor parallel elements give $y\not\in \cl{}C$, so $\cl{}C$ is
  the required cyclic flat.  The following step will therefore
  complete the proof: given $C\in \mathcal{C}(M)$ with $x\in C$ and
  $\rkm{\clm{C}\cap T}\geq 2$, we construct $C'\in
  \mathcal{C}(M)$ with $x\in C'$ and $\rkm{\clm{C'}\cap T}=1$.

  Let $C$ be as just stated.  Set $k=\rkm{C\cup T}-\rkm{T}$ and
  $F=\clm{C}$.  The fact that $T$ is a modular flat of $M$ gives $\rkm{F}-
  \rkm{F\cap T}=k$.  Let $Y\subseteq C-(T\cup x)$ be a basis of $F-T$
  in $M/(F\cap T)$.  Then $\rkm{\clm{Y}}-
  \rkm{\clm{Y}\cap T}=k$ since $|Y|=k$ and $\clm{Y}\cap
  T=\emptyset$.  Modularity then gives $\rkm{Y\cup T} -\rkm{T}=k$, from
  which $\clm{Y\cup T} = \clm{C\cup T}$ follows.  Now $\clm{Y\cup
  x}$ is a rank-$(k+1)$ flat of $M$ with $\clm{Y\cup x\cup T}$ also
  being $\clm{C\cup T}$, so modularity implies that $\clm{Y\cup
  x}\cap T$ is a point, say $p$.  The fact that $p$ is not in $\clm{Y}$
  implies that $x$ is in the fundamental circuit $C'$ of $p$ with
  respect to the basis $Y\cup x$ of $\clm{Y\cup x}$; thus, as needed,
  $C'\in \mathcal{C}(M)$, $x\in C'$, and $\rkm{\clm{C'}\cap T}=1$.
\end{proof}

The property in Proposition~\ref{pro:gpc} does not hold for amalgams
in general.  For instance, let $M$ be the irreducible matroid
$M(K_4)$, the cycle matroid of the complete graph $K_4$.  Form $N$
from $M$ by relabelling one element $a$ as $a'$.  The only amalgam
of $M$ and $N$ has the elements $a$ and $a'$ parallel, so the
amalgam is reducible.

In contrast to Proposition~\ref{pro:sum}, the converse of
Proposition~\ref{pro:gpc} can fail, even in the special case of a
parallel connection at a point.  For example, starting with the
matroid $U_{2,3}$ on $\{a,b,p\}$, form $M$ by first taking the
parallel connection, at $a$, of this matroid with a copy of $M(K_4)$,
and then take the parallel connection, at $b$, of the result with a
second copy of $M(K_4)$.  The resulting rank-$6$ matroid $M$ on the
set $A$ is reducible since every cyclic flat that contains $p$ also
contains $a$.  Relabel the elements in $A-p$ to get a matroid $N$
on $B$ isomorphic to $M$ and with $A\cap B=p$.  It is easy to
see that the parallel connection of the reducible matroids $M$ and
$N$ at $p$ is irreducible.

\subsection{Minors}

We now show that minors of free splices are also free splices.
From this point on, we adhere to the convention that all unary
operations on matroids are performed before binary operations;
so, for example, $M\fsp N|X$ means $M\fsp (N|X)$ rather than 
$(M\fsp N)|X$.  

\begin{thm}\label{thm:minor}
If $M(A)$ and $N(B)$ are matched, $i=\rkm{A-B}$, and 
$X\subseteq A\cup B$, then
  $$
  (M\fsp N)|X \= M|(X\cap A)\fssp N' \sands (M\fsp N).X \= M'\fssp 
N.(X\cap B),
  $$
  where
  $$
  N'\= \higgs j{\ist M|X.(X\cap B)}{N|(X\cap B)} \sands M'\= \higgs
  k{M.(X\cap A)}{\loo N.X|(X\cap A)},
  $$
  with $j=i-\rkm{X-B}$ and $k=i -\rkm{A-X}+\rkn{B-X}
  -|B-(A\cup X)|$.
\end{thm}

\begin{proof}
Let $L=M\fsp N$.
  Since $(A,B)$ is a free separator of $L$ and any cyclic flat of
  $L|X$ is of the form $Z\cap X$ for some $Z\in\cz (L)$, it is
  immediate from Theorem \ref{thm:freesep} that $(X\cap A, X\cap B)$
  is a free separator of $L|X$, so
  \begin{align*}
  L|X \= & \, L|X|(X\cap A)\fssp L|X.(X\cap B) \\
   \= & \, M|(X\cap  A)\fssp L|X.(X\cap B).
  \end{align*}
  Proposition \ref{pro:hminor} gives $L|X=\higgs i{\ist M}{\loo
    N}|X=\higgs i{\ist M|X}{\loo N|X}$.  Using
  Proposition \ref{pro:hminor} again, and recalling that the elements
  of $A-B$ are loops of $N_0$, gives
  \begin{align*}
    L|X.(X\cap B) \= & \, \higgs{i-\ell}{\ist M|X.(X\cap B)}{N_0|X.(X\cap
      B)} \\
    \= & \, \higgs{i-\ell}{\ist M|X.(X\cap B)}{N|(X\cap B)},
  \end{align*}
  where $\ell = \rk{\ist M|X}{X-(X\cap B)}-\rk{\loo N|X}{X-(X\cap
    B)}=\rkm{X-B}$.  Thus, $L|X.(X\cap B) = N'$\!, and so $L|X = M|(X\cap
  A)\fsp N'$. The formula for $L.X$ follows either by duality or by a
  similar application of Proposition \ref{pro:hminor}.
\end{proof}

We next characterize the special cases in which restriction and
contraction of free splices can be expressed in the simplest way.
Mildly extending the definition often seen for flats, we call a pair
  $(X,Y)$ of subsets of $E$ a \emph{modular pair} in the matroid
  $M(E)$ if $\rkm{X} + \rkm{Y} = \rkm{X\cup Y} + \rkm{X\cap Y}$, that
  is, if equality holds in the semimodular inequality.

\begin{cor}\label{cor:minor1}
  Assume $M(A)$ and $N(B)$ are matched and $X\subseteq A\cup B$. We
  have $(M\fsp N)|X=M|(X\cap A)\fssp N|(X\cap B)$ if and only if
  either
  \begin{itemize}
  \item[(a)] $\rkm{A-B}=\rkm{X-B}$, 
that is, $A-B\subseteq \clm{X-B}$, or
  \item[(b)] all elements of $X-A$ are isthmuses of $N|(X\cap B)$ and
    $(A-B,X\cap A)$ is a modular pair in $M$.
  \end{itemize}
  Likewise, $(M\fsp N).X=M.(X\cap A)\fssp N.(X\cap B)$ if and only if
  either
  \begin{itemize}
  \item[(a$'$)] all elements of $(B-A)-X$ are isthmuses of
    $N|(B-(X-A))$, or
  \item[(b$'$)] all elements of $X-B$ are loops of $M.(X\cap A)$ and
    $(A\cap B,B-X)$ is a modular pair in $N$.
  \end{itemize}
\end{cor}

\begin{proof} 
Let $j=\rkm{A-B}-\rkm{X-B}$. By Theorem \ref{thm:minor},
the expression for $(M\fsp N)|X$ in this corollary is valid
if and only if $\higgs j{\ist M|X.(X\cap B)}{N|(X\cap B)}= N|(X\cap B)$.
Since $j\geq 0$, this is the case
if and only
either $j=0$, which is just statement (a), or $M_1|X.(X\cap B)= N|(X\cap B)$.
Since all elements of $B-A$ are isthmuses of
$M_1$, this last equation holds if and only if all elements of $X-A$ are
isthmuses of $N|(X\cap B)$ and
$$
     M_1|X.(X\cap B)|(X\cap A\cap B)\= N|(X\cap A\cap B).
$$  
Simplifying and using the matching condition $N|(A\cap B)=M.(A\cap B)$,
this equation can be rewritten as 
   $$M|(X\cap A).(X\cap
   A\cap B) \= M.(A\cap B)|(X\cap A\cap B),
   $$
which says that in $M\del ((A\cap B)-X)/(X-(A\cap B))$,
the same minor results whether $A-(B\cup X)$ is deleted or
contracted.  Using standard results on connectivity and separators
(see, e.g.,~\cite[Section 4.2]{ox:mt}), this statement can be
recast as saying that $(A-B, X\cap A)$ is a modular pair of sets in
$M$.  The assertion about contractions follows by duality, noting
that the containment $B-A\subseteq \cl{N^*}{X-A}$ holds if and only if
$(B-A)-X\subseteq \isthm{N|(B-(X-A))}$.
\end{proof}

\begin{cor}\label{cor:minor2}
Suppose that $M(A)$ and $N(B)$ are matched and that 
$X\subseteq A\cup B$.
\begin{enumerate}[\rm (1)\,]
\item
If $A-B\subseteq X$, then $(M\fsp N)|X = M|(X\cap A) \fssp N|(X\cap B)$.
\item
If $B-A\subseteq X$, then $(M\fsp N).X = M.(X\cap A) \fssp N.(X\cap B)$.
\end{enumerate}
\end{cor}

\subsection{Free product}

As mentioned in Section \ref{subsec:free}, whenever the ground sets of
$M(A)$ and $N(B)$ are disjoint, the free splice of $M$ and $N$ is the
free product $M\frp N$.  The various cryptomorphic descriptions of
free product given in \cite{crsc:uft} are obvious specializations of
their free splice counterparts, given in Theorem \ref{thm:crypt}.

It was shown in \cite{crsc:pem} that in the case of disjoint ground
sets, the set $\splice MN=\{L\co\text{$L|A=M$ and $L.B=N$}\}$ is
the interval $[M\oplus N, M\frp N]$ in the weak order on $\cm (A\cup B)$.
(Recall that, as we saw in Example \ref{exa:mins}, the set \splice MN\
is not in general an interval, since it may have many minimal
elements.)

Suppose $M(A)$ and $N(B)$ are given, where $A$ and $B$ need not be
disjoint, and denote by  $I_1$ and $I_2$ the intervals
$I_1 \= [M|(A-B) \oplus N,\, M|(A-B) \frp N]$ and $I_2\= [M\oplus
N.(B-A),\, M\frp N.(B-A)]$ in $\cm (A\cup B)$.  We then have
$$
I_1\= \{ L\co\text{$L|(A-B)=M|(A-B)$ and $L.B=N$}\}
$$
and
$$
I_2 \= \{ L\co\text{$L|A=M$ and $L.(B-A)=N.(B-A)$}\},
$$
and hence it follows that $\splice MN = I_1\cap I_2$.  Note that when
$A$ and $B$ are disjoint we have $I_1=I_2=[M\oplus N,\, M\frp N] =
\splice MN$.

Suppose that $L$, $P$, and $Q$ are matroids on the same set. The matroid $L$ is
the \emph{intersection} of $P$ and $Q$ if $\ci (L) = \ci (P)\cap\ci
(Q)$. 
When $P$ and $Q$ have the same rank, this is the case if and only if
$\cb (L)=\cb (P)\cap\cb (Q)$.  (For arbitrary matroids $P$ and $Q$ on
the same set, the intersection $\ci (P)\cap\ci (Q)$ is not the
collection of independent sets of a matroid.)   We now show that
a free splice is the intersection of two free products.

\begin{pro}
  If $M(A)$ and $N(B)$ are matched, then the free splice $M\fsp N$ is
  the intersection of the matroids $M|(A-B)\frp N$ and $M\frp
  N.(B-A)$.
\end{pro}

\begin{proof}
  Let $r=\rkm{A-B} + \r N$.  By Theorem \ref{thm:crypt}, the
$r$-element subsets $X$ of $A\cup B$ such that
$X-B\in\ci (M)$ and $X\cap B\in\cs (N)$ are the bases of
of $M|(A-B)\frp N$, and those satisfying 
$X\cap A\in\ci (M)$ and $X- A\in\cs (N.(B-A))$, that is,
$X\cap A\in\ci (M)$ and $(X- A)\cup (A\cap B)\in\cs (N)$,
are the bases of $M\frp N.(B-A)$. The containment
$\cb (M|(A-B)\frp N)\cap\cb (M\frp N.(B-A))
\subseteq\cb (M\fsp N)$ is thus apparent.
For the reverse inclusion, it suffices to note
  that, for all $X\subseteq A\cup B$, we have $X-B\subseteq X\cap A$ and
  $X\cap B\subseteq (X-A)\cup (A\cap B)$, and so $X\cap A\in\ci (M)$
  implies that $X-B\in\ci (M)$, and $X\cap B\in\cs (N)$ implies that
  $(X- A)\cup (A\cap B)\in\cs (N)$.
\end{proof}

\section{Algebraic properties of the free splice
operation}\label{sec:assoc}

Unlike free product and direct sum, free splice is a nonassociative
operation.  In this section, we characterize the triples of matroids
for which associativity holds.  Furthermore, in Theorem
\ref{thm:assoc} below we show that a weakened version of associativity
holds in general.  The key to the proof of this result is a basic property of
free separators, given in the following lemma.
\begin{lem}\label{lem:fsepassoc}
  Let $L$ be a matroid on $A\cup B\cup C$.
  \begin{enumerate}[\rm (1)\,]
  \item If $(A\cup B, C)\in \fs L$ and $(A,B)\in\fs{L|(A\cup B)}$,
    then $(A, B\cup C)\in\fs L$.
  \item If $(A, B\cup C)\in\fs L$ and $(B,C)\in\fs{L.(B\cup C)}$, then
    $(A\cup B, C)\in\fs L$.
  \end{enumerate}
\end{lem}
\begin{proof}
  Suppose that $(A\cup B, C)\in \fs L$ and $(A,B)\in\fs{L|(A\cup B)}$.
  By Theorem \ref{thm:crypt} we have 
$$\cz (L)\,\subseteq\,\cz (L|(A\cup
  B))\cup\, [(A\cup B)-C, A\cup B\cup C ]
$$ 
and 
$$\cz (L|(A\cup
  B))\subseteq\, [\emptyset, A]\cup [A-B, A\cup B].
$$  
Since $A-(B\cup
  C)\subseteq (A\cup B)-C$ and $A-(B\cup C)\subseteq A-B$, it follows
  that $\cz (L)\subseteq [\emptyset , A]\cup [A-(B\cup C), A\cup B\cup
  C ]$, and hence Theorem \ref{thm:freesep} implies that $(A,B\cup
  C)\in\fs L$.  Statement (2) follows by duality.
\end{proof}

\begin{thm}\label{thm:assoc}
  Suppose $M(A)$, $N(B)$, and $P(C)$ are matroids. Let $U= A\cap
  (B\cup C)$ and $V= (A\cup B)\cap C$.
  \begin{enumerate}[\rm (1)\,]
  \item If $(M,N)$ and $(M\fsp N, P)$ are matched, then $(M\fsp N)\fsp
    P = M\fsp (N'\fsp P)$, where $N' = M.U\fsp N$.
  \item If $(N,P)$ and $(M,N\fsp P)$ are matched, then $M\fsp (N\fsp
    P) = (M\fsp N'')\fsp P$, where $N'' = N\fsp P|V$.
  \end{enumerate}
\end{thm}

\begin{proof}
  Let $L=(M\fsp N)\fsp P$.  Since $(A,B)\in\fs{M\fsp N}=\fs{L|(A\cup
    B)}$ and $(A\cup B, C)\in\fs L$, Lemma \ref{lem:fsepassoc} implies
  that $(A, B\cup C)\in\fs L$, that is, $L=L|A\fsp L.(B\cup C)$.
  Now $L|A=L|(A\cup B)|A=(M\fsp N)|A=M$, and Corollary
  \ref{cor:minor2} gives
  \begin{align*}
    L.(B\cup C) &\= (M\fsp N).((A\cup B)\cap (B\cup C))\fssp P\\
    &\= (M.(A\cap ((A\cup B)\cap (B\cup C)))\fssp N)\fssp P\\
    &\= (M.U\fsp N)\fssp P,
  \end{align*}
  and hence (1) follows.  Statement (2) follows by duality.
\end{proof}

\begin{cor}\label{cor:absorb}
  Suppose that $M(A)$, $N(B)$, and $P(C)$ are matroids.
  \begin{enumerate}[\rm (1)\,]
  \item If $(M,N)$ and $(M\fsp N, P)$ are matched and $B\subseteq C$,
    then $(M\fsp N)\fsp P = M\fsp P$.
  \item If $(N,P)$ and $(M, N\fsp P)$ are matched and $B\subseteq A$,
    then $M\fsp (N\fsp P) = M\fsp P$.
  \end{enumerate}
\end{cor}

\begin{proof}
  If $B\subseteq C$, then the ground set of the matroid $N'$ in
  Theorem \ref{thm:assoc} is contained in $C$. 
  As noted after Theorem \ref{thm:freest}, it follows that $N'\fsp
  P = P$, and so (1) follows by Theorem \ref{thm:assoc}.  Statement
  (2) follows by duality.
\end{proof}

We now consider the simplest situation in which free splice is
associative.

\begin{pro}\label{pro:emptyassoc}
  If $M(A)$, $N(B)$, and $P(C)$ are matroids with $A\cap C\subseteq
  B$, then the following are equivalent:
  \begin{itemize}
  \item[(1)] $(M,N)$ and $(N,P)$ are matched,
  \item[(2)] $(M,N)$ and $(M\fsp N, P)$ are matched,
  \item[(3)] $(N,P)$ and $(M, N\fsp P)$ are matched.
  \end{itemize}
  Furthermore, if these conditions are satisfied, then $(M\fsp N)\fsp
  P =M\fsp (N\fsp P)$.
\end{pro}

\begin{proof}
  Suppose $(M,N)$ is matched.  Since $A\cap C\subseteq B$, we have
  $(A\cup B)\cap C = B\cap C$, and so the pair $(M\fsp N,P)$ being
  matched means that $(M\fsp N).(B\cap C) =P|(B\cap C)$.  But $(M\fsp
  N).(B\cap C)=(M\fsp N).B.(B\cap C)=N.(B\cap C)$, and thus $(M\fsp N,
  P)$ is matched if and only if $N.(B\cap C)=P|(B\cap C)$.  Hence (1)
  and (2) are equivalent.  The equivalence of (1) and (3) follows by
  duality.

  By Theorem \ref{thm:assoc}, statement (2) implies that $(M\fsp
  N)\fsp P = M\fsp (N'\fsp P)$, where $N' = M.(A\cap (B\cup C))\fsp
  N$.  The fact that $A\cap (B\cup C)\subseteq B$ then gives $N'=N$.
\end{proof}

The next proposition is a type of commutativity result, showing when a
matroid occurs as both a left and a right factor in a free splice.
This result is essential to the proof of Theorem \ref{thm:associate},
which characterizes associative triples of matroids.

\begin{pro}\label{pro:com}
  Suppose $L(E)$ is a matroid with $E=A\cup B= B\cup C$.  Let
  $S=A-B=C-B$.  The pairs $(A,B)$ and $(B,C)$ are free separators of
  $L$ and $L|B=L.B$ if and only if either $S=\emptyset$ or one of the
  statements (a)--(c) holds:
  \begin{itemize}
  \item[(a)] $\;S\cup (B-A)\subseteq\isthm L$,
  \item[(b)] $\;S\cup (B-C)\subseteq\loops L$,
  \item[(c)]
    \begin{enumerate}[\rm (i)]
    \item $\;B\subseteq A\cup C$,
    \item $\;B-A\subseteq\isthm L$ and $B-C\subseteq\loops L$, and
    \item $(S,A\cap B\cap C)$ is a modular pair in $L$.
    \end{enumerate}
  \end{itemize}
\end{pro}

\begin{proof}
  The assertion is obvious if $S=\emptyset$, so assume $S\ne
  \emptyset$.  Suppose $L|B=L.B$ and $(A,B),(B,C)\in\fs L$.  Thus,
  $L=L|A\fsp L.B$.  Since $L|B=L.B$ and $A\cap B\subseteq B$, we
  have $L|B=L|(A\cap B)\fsp L.B|B$, so, by Corollary
  \ref{cor:minor1} with $X=B$, one of the following statements holds:
  \begin{itemize}
  \item[(1)] $r_{_{L|A}}(S)=0$, so all elements of $S$ are loops of
    $L$, or
  \item[(2)] all elements of $B-A$ are isthmuses of $L.B$ and so of
    $L$; also, $(S,A\cap B)$ is a modular pair in $L|A$ and so in $L$.
  \end{itemize}
  Similarly, $L=L|B\fsp L.C$ and $L.B=L|B.B\fsp L.(B\cap C)$, so
  Corollary \ref{cor:minor1} implies that one of the following
  statements holds:
  \begin{itemize}
  \item[(1$'$)] all elements of $S$ are isthmuses of $L.C$ and so of
    $L$, or
  \item[(2$'$)] all elements of $B-C$ are loops of $L|B$ and so of
    $L$; also, $(B\cap C,S)$ is a modular pair in $L.C$, that is,
    $(B,S\cup (B-C))$ is a modular pair in $L$.
  \end{itemize}
  Since $S$ is nonempty, statements (1) and (1$'$) are incompatible.
  Statements (2) and (2$'$) can both hold only if $B\subseteq A\cup
  C$; furthermore, in this case the assertions about loops and
  isthmuses reduce the modularity conditions to one, namely, that in
  part (iii) of (c).  Note also that the modularity assertion in
  statement (2$'$) would follow immediately from statement (1); a
  similar remark applies to statements (1$'$) and (2). Hence one of
  the statements (a)--(c) holds.

  Conversely, when (a) holds, each cyclic flat of $L$ is contained in
  $A\cap B$; when (b) holds, each cyclic flat contains $B-C$ and $S$; and 
  when (c) holds, each cyclic flat contains $B-C$ and is contained in
  $A$. In each case, Theorem \ref{thm:freesep} implies that
  $(A,B),(B,C)\in\fs L$.  When (a) or (b) holds, the fact that the elements of
  $S$ are loops or isthmuses yields $L|B=L.B$; when (c) holds, (ii)
  and (iii) yield $L|B=L.B$.
\end{proof}

For convenience, we give the following restatement of Proposition
\ref{pro:com}.

\begin{pro}\label{pro:comrecast}
  Let $M(A)$, $N(B)$, and $P(C)$ be matroids with $A\cup B = B\cup C$.
  Let $S=A-B=C-B$ and $T=A\cap B\cap C$.  The equality $M\fsp N=N\fsp
  P$ holds if and only if either $S=\emptyset$ or one of the statements
  (a)--(c) holds:
  \begin{itemize}
  \item[(a)] $M= I(S)\oplus Q$,\; $N= Q\oplus I(B-A)$,\; and\;\;
    $P= I(S)\oplus Q.T\oplus I(C-A)$,\\
    for some matroid $Q(A\cap B)$,\;\;
  \item[(b)] $M= I^*(S)\oplus R|T\oplus I^*(A-C)$,\; $N= R\oplus
    I^*(B-C)$,\; and\;\;
    $P= I^*(S)\oplus R$,\\
    for some $R(B\cap C)$, or
  \item[(c)] $M=Q(S)\oplus R(T)\oplus I^*(A-C)$,\; $N=R(T)\oplus
    I^*(A-C)\oplus I(C-A)$,\; and $P=Q(S)\oplus R(T)\oplus I(C-A)$,
    for some $Q(S)$ and $R(T)$.
  \end{itemize}
\end{pro}

As a special case, we have the following characterization of the pairs
for which free splice is commutative.
\begin{cor}
  If $M(A)$ and $N(B)$ are matched, then $M\fsp N=N\fsp M$ if and only
  if one of the following conditions holds:
  \begin{itemize}
  \item[(a)] $A-B\subseteq\loops M$ and $B-A\subseteq\loops N$,
  \item[(b)] $A-B\subseteq\isthm M$ and $B-A\subseteq\isthm N$,
\item[(c)] $B\subseteq A$ and $M=M|B\oplus M|(A-B)$,
\item[(d)] $A\subseteq B$ and $N=N|A\oplus N|(B-A)$.
  \end{itemize}
\end{cor}

We now characterize the triples for which free splice is associative.

\begin{thm}\label{thm:associate}
  Suppose that $M(A)$, $N(B)$, and $P(C)$ are matroids and that the
  pairs $(M,N)$, $(M\fsp N, P)$, $(N,P)$, and $(M, N\fsp P)$ are
  matched.  The equality $(M\fsp N)\fsp P = M\fsp (N\fsp P)$ holds if
  and only if either $(A\cap C)-B=\emptyset$ or one of statements
  (a)--(c) holds:
  \begin{itemize}
  \item[(a)] $(A\cap C)-B\subseteq\isthm M$ and $B-A\subseteq\isthm
    N$,
  \item[(b)] $(A\cap C)-B\subseteq\loops P$ and $B-C\subseteq\loops
    N$,
  \item[(c)]
    \begin{itemize}
    \item[(i)] $B\subseteq A\cup C$,
    \item[(ii)] $B-A\subseteq\isthm N$ and $B-C\subseteq\loops N$, and
    \item[(iii)] $((A\cap C)-B, A\cap B\cap C)$ is a modular pair in
      $P$.
    \end{itemize}
  \end{itemize}
\end{thm}

\begin{proof}
  Proposition \ref{pro:emptyassoc} treats the case $A\cap C\subseteq
  B$, so assume $(A\cap C)-B\ne\emptyset$.  Let $U= A\cap (B\cup C)$
  and $V= (A\cup B)\cap C$.  The hypotheses and Theorem
  \ref{thm:assoc} give 
$$
(M\fsp N)\fssp P \= M\fssp ((M.U\fsp N)\fssp P).
$$ 
Also, $N\fsp P = (N\fsp P|V)\fsp P$ by Corollary
\ref{cor:absorb} and the hypotheses, so
$$
M\fssp (N\fsp P) \= M\fssp ((N\fsp P|V)\fssp P).
$$
Hence $(M\fsp N)\fsp P= M\fsp (N\fsp P)$ is equivalent to
$$
M\fssp ((M.U\fsp N)\fssp P)\= M\fssp ((N\fsp P|V)\fssp P).
$$
Since the underlying free separators are the same, it follows that
this equation is equivalent to $(M.U\fsp N)\fsp P= (N\fsp P|V)\fsp
P$ and, in turn, to
\begin{equation}\label{eq:target}
    M.U\fssp N\=N\fssp P|V. 
\end{equation}
If this equation holds, then one of statements (a)--(c) in Proposition
\ref{pro:com} holds; in a straightforward manner, these statements
imply, respectively, statements (a)--(c) above.

For the converse, first assume that (a) holds.  Let $S=U-B=V-B$ and
$T=U\cap B\cap V$.  Since $B-A\subseteq\isthm N$, Corollary
\ref{cor:loopisth} implies that all isthmuses of $M$ are isthmuses of
$M\fsp N$.  Since, in addition, $(M\fsp N).V=P|V$ (because the pair
$(M\fsp N,P)$ is matched), the hypothesis $S\subseteq \isthm M$ gives
$S\subseteq \isthm {P|V}$.  Setting $Q=M.(A\cap B)=N|(A\cap B)$ now
gives 
$$M.U\=I(S)\oplus Q, \quad N\=Q\oplus I(B-V), \quad \text{and}
\quad P|V\=I(S)\oplus Q.T\oplus I(V-U).
$$
From this and part (a) of
Proposition \ref{pro:comrecast}, Equation~(\ref{eq:target}) follows,
as needed.  Statement (b) is handled similarly or via duality.  Now
assume statement (c) holds.  Note that $(M,P)$ is matched since
$(M\fsp N,P)$ is matched and $B\subseteq A\cup C$.  By (iii), the
matroid $M.(A\cap C)=P|(A\cap C)$ is $Q(S)\oplus R(T)$ for some
matroids $Q$ and $R$.  Since $(M\fsp N).V=P|V$, statement (ii) gives
$B-A\subseteq \isthm{P|V}$, that is, $V-U\subseteq \isthm{P|V}$.
Similarly, $U-V\subseteq \loops{M.U}$.  These observations give
$$
M.U\=Q(S)\oplus R(T)\oplus I^*(U-V), \qquad N\=R(T)\oplus
I^*(U-V)\oplus I(V-U),
$$
and
$$
P|V\=Q(S)\oplus R(T)\oplus I(V-U).
$$
Thus, part (c) of Proposition \ref{pro:comrecast} applies and gives
Equation~(\ref{eq:target}), thereby completing the proof.
\end{proof}

\section{Splices and classes of matroids}
\label{sec:class}

Given a class $\mathcal{C}$ of matroids, it is natural to ask whether
$\mathcal{C}$ contains all splices, the free splice, or at least one
splice of any two matched matroids in $\mathcal{C}$.  Much of this
section shows that even the weakest of these questions has a positive
answer for few of the commonly-studied classes of matroids.  Also, we
show that even the simplest nontrivial class of matroids that is
generated by the free splice --- that obtained by starting with loops
and isthmuses and iteratively taking the free splice --- is huge and
has some striking properties, notably its failure to be minor-closed.
These results may make one wonder whether any nontrivial minor-closed
class of matroids is closed under free splice; to address this
question, we identify sufficient conditions for the excluded minors
that guarantee that the corresponding minor-closed class of matroids
is also closed under the free splice, and we show that, for ranks
three and greater, binary projective geometries and cycle matroids of
complete graphs, as well as their duals, satisfy these conditions.

\subsection{Representable and algebraic matroids} 

We start by considering representable matroids.  Crapo and
Schmitt~\cite{crsc:uft} showed that the free product of two matroids that
are representable over a given field is representable over every
sufficiently large field of the same characteristic; however, the
class of matroids that are representable over a given finite field is
not closed under free product.  In contrast, examples below show
that the free splice of matroids that are representable over a given
field might not be representable over any field.  However, there is a
positive result for splices of binary and ternary matroids, the key to
which is unique representability.  Recall that two matrix
representations of a matroid over a field $F$ are \emph{equivalent} if
one can be obtained from the other by the following operations:
interchange two rows; interchange two columns (along with their
labels); multiply a row or a column by a nonzero element of $F$;
replace a row by its sum with another row; replace every matrix entry
by its image under an automorphism of $F$; and delete or adjoin rows
of zeroes.  A matroid that is representable over $F$ is \emph{uniquely
  $F$-representable} if all of its matrix representations over $F$ are
equivalent.  Brylawski and Lucas, who introduced this important idea
in ~\cite{brlu:urc}, proved that a binary matroid is uniquely representable
over every field over which it is representable; also, ternary
matroids are uniquely representable over $\GF(3)$.  It follows that if
a matroid $M$ on $A$ is binary (resp., ternary) and if $P$ is a binary
(resp., ternary) matrix that represents a restriction $M|S$, then
there is a binary (resp., ternary) matrix that represents $M$ and for
which the columns corresponding to the elements in $S$ form $P$,
possibly with some added rows of zeroes.
\begin{pro}\label{pro:ur}
  Fix $F\in \{\GF(2),\GF(3)\}$.  If the matched matroid $M(A)$ and
  $N(B)$ are representable over $F$, then so is some splice of $M$ and
  $N$.
\end{pro}

\begin{proof}
  Let $X$ be a basis of $M|(A-B)$; let $|X|=k$.  Thus, $M$ has a
  matrix representation over $F$ of the form
  $$\left(
    \begin{array}{c|c|c}
      I_k & R & S \\
      \hline
      0 & 0 & T 
    \end{array} \right)$$
  where the columns of the identity matrix $I_k$ correspond to the
  elements of $X$, the columns of $R$ correspond to the elements of
  $(A-B)-X$, and the zeroes denote the zero matrices of the
  appropriate sizes.  The matrix realization of contraction shows that
  $T$ is a representation, over $F$, of $M.(A\cap B)$, that is, $N|
  (B\cap A)$, so unique representability implies that $N$ has a
  representation over $F$ of the form
  $$\left(
    \begin{array}{c|c}
      T & U  \\
      \hline
      0 & V  
    \end{array} \right).$$
  It follows that the matroid represented over $F$ by the matrix 
  $$\left(
    \begin{array}{c|c|c|c}
      I_k & R & S & W \\
      \hline
      0 & 0 & T & U \\ 
      \hline
      0 & 0 & 0 & V  
    \end{array} \right),$$ 
  for any matrix $W$ over $F$ of the appropriate size,
  is a splice of $M$ and $N$.
\end{proof}

We note that it is possible for the first two displayed matrices in
the proof above to be totally unimodular without the third having this
property, even if $W=0$.  Thus, if the counterpart of this result is
true for regular matroids, a different approach to the proof would be
needed.

\begin{figure}
\begin{center}
\includegraphics[width = 3.6truein]{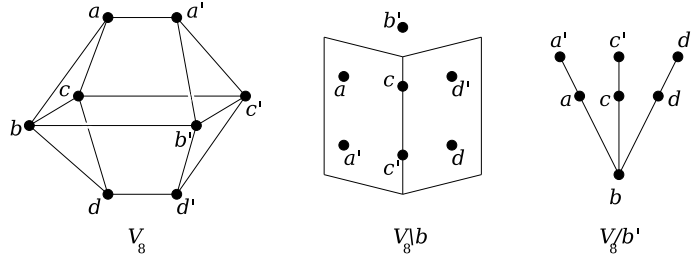}
\end{center}
\caption{The V\'amos cube $V_8$ and two minors whose free splice is
  $V_8$.}\label{vamos}
\end{figure}

The following example stands in contrast to the result above.  The
V\'amos cube, $V_8$, is the rank-$4$ matroid on
$\{a,a',b,b',c,c',d,d'\}$ in which the proper, nonempty cyclic flats
(all of rank $3$) are all sets of the form $\{x,x',y,y'\}$ except
$\{a,a',d,d'\}$ (see Figure~\ref{vamos}).  This matroid is neither
representable nor algebraic over any field.  Since $b$ and $b'$ are
clones, Theorem~\ref{thm:freesep} implies that $V_8$ is the free
splice of $V_8\del b$ and $V_8/b'$, both of which are representable
over all fields except $\GF(2)$ and $\GF(3)$; also, both are gammoids
and both are algebraic over all fields.  Furthermore, the geometric
argument that $V_8$ is not representable over any field also shows
that no splice of $V_8\del b$ and $V_8/b'$ is representable over any
field; likewise, the argument Ingleton and Main~\cite{inma:nam} give
to show that $V_8$ is not algebraic applies to all splices.  (Splices
other than the free splice have additional cyclic planes; for
instance, $\{a,b,c,d\}$ might be a cyclic plane.)  It follows that
there is no counterpart of Proposition~\ref{pro:ur} for any field other
than $\GF(2)$ and $\GF(3)$; also, no result of that type applies to
the class of gammoids or to the class of matroids that are algebraic
over any fixed field.

\subsection{Matroids with no $U_{2,q+2}$-minor}

For an integer $q>1$, let $\mathcal{U}(q)$ be the class of matroids
that have no $U_{2,q+2}$-minor.  These classes arise often in
extremal matroid theory.  Note that $\mathcal{U}(2)$ is the class of
binary matroids, but if $q$ is any other prime power, then
$\mathcal{U}(q)$ properly contains the class of matroids that are
representable over $\GF(q)$.  We now show that for $q>2$, the
counterpart of Proposition~\ref{pro:ur} fails for $\mathcal{U}(q)$.  Let
$p$ be the largest prime power less than $q$.  The projective plane
$\PG(2,p)$ is in $\mathcal{U}(q)$; let this be $M$.  Fix a point $a$
in the ground set $A$ of $M$ and an element $x\not\in A$.  Let $N$ be
the rank-$2$ matroid on $(A-a)\cup x$ whose rank-$1$ flats are $x$ and
the sets $\ell-a$ as $\ell$ runs over the lines of $M$ with $a\in
\ell$.  Note that $M$ and $N$ are matched; also, $N\in\mathcal{U}(q)$.
Any splice of $M$ and $N$ extends $M$ by putting $x$ either freely in
$M$ or freely on (only) one of the lines of $M$ not containing $a$;
neither type of splice is in $\mathcal{U}(q)$ since the former has
$p^2+p+1$ lines through $x$ while the latter has $p^2+1$, and both
numbers exceed $q+1$.

\subsection{Transversal matroids}

  We now show that the counterpart of Proposition~\ref{pro:ur} fails for
  the class of transversal matroids, and even for the more restricted
  classes of fundamental transversal matroids and bicircular matroids.
  First recall the geometric representation of transversal matroids
  from~\cite{br:art}: a matroid is transversal if and only if it has a
  geometric representation on a simplex in which each cyclic flat of
  rank $k$ spans a $k$-vertex face of the simplex.  Fundamental
  transversal matroids have such a representation in which, in
  addition, there is an element of the matroid at each vertex of the
  simplex.  Bicircular matroids have such a representation in which
  each element is along an edge or at a vertex of the simplex.  Now
  consider the $3$-whirl on $\{c,c',d,e,f,g\}$ with non-spanning
  circuits $\{e,f,g\}$, $\{c',d,g\}$, and $\{c,d,e\}$; add an isthmus
  $a$ to get the matroid $M$, which is both a fundamental transversal
  matroid and a bicircular matroid.  In any simplex representation of
  $M$, the elements $d,e,g$ are at vertices of the simplex.  Let $N$
  be the rank-$3$ matroid on $\{a,b,b',c',d,e,f,g\}$ whose proper
  nonempty cyclic flats are $\{d,e\}$ of rank $1$, and, of rank $2$,
  $\{a,b,b',c'\}$ and $\{c',d,e,f,g\}$.  Note that $N$ is also a
  fundamental transversal matroid and a bicircular matroid; also, $M$
  and $N$ are matched.  Since $\{a,b,b',c'\}$ is a cyclic flat of $N$,
  it follows that in any splice of $M$ and $N$, either $\{a,b,b',c'\}$
  or $\{a,b,b',c,c'\}$ is a cyclic flat; however, a cyclic flat in a
  transversal matroid must intersect the face $\{c,c',d,e,f,g\}$ of
  the simplex in some face, not, for instance, at $c$ and $c'$.  Thus,
  no splice of $M$ and $N$ is transversal.

\subsection{Base-orderable matroids}

We now show that the free splice need not preserve the property of
being base-orderable.  We first recall some definitions.  For bases
$B$ and $B'$ of a matroid $M$, elements $x\in B$ and $x'\in B'$ are
\emph{exchangeable} if both $(B-x)\cup x'$ and $(B'-x')\cup x$ are
bases of $M$.  A matroid $M$ is \emph{base-orderable} if for each pair
of bases $B$ and $B'$ of $M$, there is a bijection $\phi:B\rightarrow
B'$ such that for all $x\in B$, the elements $x$ and $\phi(x)$ are
exchangeable.  It is known that the class of base-orderable matroids
is closed under minors, duals, free extension, and matroid union; from
this and Joseph Kung's observation that free products can be expressed
as certain matroid unions (see his review of~\cite{crsc:uft} in
Mathematical Reviews, MR2177484), it follows that the class of
base-orderable matroids is closed under free product.  It is also
known (and follows from the results just mentioned) that all gammoids
are base-orderable.  The matroid $M$ in Figure~\ref{bo} (which appears
in~\cite{in:nbo} and is used to illustrate the theory developed
there) is not base-orderable since, for $B=\{a,b,c,d\}$ and
$B'=\{s,t,u,v\}$, the elements $a$ and $b$ are each exchangeable only
with $u$.  Note that $c$ and $d$ are clones, so $M$ is the free splice
of $M\del c$ and $M/d$, which are transversal and so are
base-orderable.

\begin{figure}
\begin{center}
\includegraphics[width = 1.2truein]{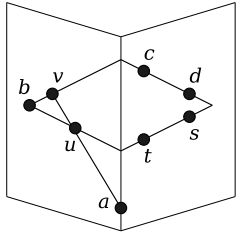}
\end{center}
\caption{A rank-$4$ matroid $M$ that is not base-orderable, yet is
  the free splice of its base-orderable minors $M\del c$ and
  $M/d$.}\label{bo}
\end{figure}

\subsection{The smallest nontrivial class that is closed under
  free splice}

We next highlight the complexity of the free splice operation by
considering the smallest nontrivial class of matroids that it
generates.  First, recall that the matroids obtained by starting with
loops and isthmuses and repeatedly taking free products are the nested
matroids (called freedom matroids in~\cite{crsc:uft}), which can be
characterized as the matroids whose cyclic flats form a chain under
inclusion.  Up to isomorphism, there are $2^n$ nested matroids on $n$
elements; they form a small subclass of the class of fundamental
transversal matroids; also, minors of nested matroids are nested.  Now
consider the counterpart for free splice: let $\mathcal{N}$ be the
class of matroids obtained by starting with loops and isthmuses and
repeatedly taking free splices.  We show that the resulting class is
much larger than the class of nested matroids; indeed, it includes all
fundamental transversal matroids, and much more.  The proof of the
next result uses the following sufficient condition for a collection
$\mathcal{C}$ of matroids to be contained in $\mathcal{N}$: if each
matroid $L(E)$ in $\mathcal{C}$, 
with $|E|\geq 2$, can be written as $M\fsp N$ where $M$ and
$N$ are proper minors of $L$ that are in $\mathcal{C}$, then
$\mathcal{C}\subseteq\mathcal{N}$.  (Note that $\mathcal{C}$ does not
have to be minor-closed.)

\begin{pro}\label{pro:lp}
  All fundamental transversal matroids are in $\mathcal{N}$.
  Furthermore, $\mathcal{N}$ is closed under taking duals but not
  under taking minors.
\end{pro}

\begin{proof}
  A fundamental transversal matroid $L$ of rank $k$ has a geometric
  realization on a $k$-vertex simplex in which there is at least one
  element of $L$ at each vertex.  Let $v_1,v_2,\ldots, v_k$ be at the
  $k$ vertices.  Easy arguments show that we may assume that $L$ has
  no loops and that it has elements in addition to $v_1,v_2,\ldots,
  v_k$.  Let $x$ be an element of $L$ not among $v_1,v_2,\ldots, v_k$.
  Thus, $x$ is freely in some face $F$ of the simplex.  It follows
  that the elements of $L$ in $F$ form a cyclic flat and that this
  cyclic flat is contained in every cyclic flat that contains $x$.
  Therefore $L=(L\del x)\fsp(L/v_i)$ for any $v_i$ in $F$.  Since
  $v_i$ is at a vertex of the simplex and $x$ is not, both $L/v_i$ and
  $L\del x$ are fundamental transversal matroids.  The assertion about
  these matroids now follows from the sufficient condition noted
  above.

 The class $\mathcal{N}$ is clearly closed under duality.  To see
  that it is not closed under minors, consider the bicircular matroid
  $B_n$ that has one element at each vertex of the $n$-vertex simplex
  and one element freely on each edge of the simplex.  Thus, $B_n$ has
  $n+\binom{n}{2}$ elements.  Since $B_n$ is a fundamental transversal
  matroid, it is in $\mathcal{N}$.  Let $B'_n$ be the restriction of
  $B_n$ to the $\binom{n}{2}$ elements that are freely on the edges of
  the simplex.  Note that if $n\geq 5$, then, for each vertex $v$ of
  the simplex, the $\binom{n-1}{2}$ elements that are freely on the
  edges of the simplex that do not contain the vertex $v$ form a
  cyclic flat of $B'_n$.  It follows from this observation and
  Corollary~\ref{cor:bfsep} that $B'_n$ is irreducible for $n\geq 5$,
  and so is not in $\mathcal{N}$.
\end{proof}

We also note that since every lattice path matroid~\cite{bodeno:lpm}
contains either a loop, an isthmus, or a pair of clones, and since
this class of matroids is closed under minors, it follows that all
lattice path matroids are in $\mathcal{N}$.  (Although they are
transversal matroids, not all lattice path matroids are fundamental
transversal matroids.)  Not all matroids in $\mathcal{N}$ are
transversal: $V_8$ is in $\mathcal{N}$, as is the truncation of
$U_{1,2}\oplus U_{1,2}\oplus U_{1,2}$ to rank $2$.

\subsection{A sufficient condition for a minor-closed class to be
  closed under free splice}

We now turn to classes of matroids that yield positive answers to a
question posed at the start of this section: Theorem~\ref{thm:k4}
gives sufficient conditions for the excluded minors of a minor-closed
class $\mathcal{C}$ of matroids so that the free splice of any two
matched matroids in $\mathcal{C}$ will also be in $\mathcal{C}$.  We
then show that, for ranks three and greater, binary projective
geometries and cycle matroids of complete graphs, as well as their
duals, satisfy these sufficient conditions.  We start with a lemma.

\begin{lem}\label{lemma:k4}
  Assume $G$ is a matroid on the ground set $Z\cup a$ where $a\not\in
  Z$ and $a$ is neither a loop nor an isthmus of $G$.  If the ground
  set of a matroid $K$ is the disjoint union of $X$, $Y$, and $Z$, and
  if $K|Z=G\del a$ and $K.Z=G/a$, then either $K/X$ or $K\del Y$ has a
  minor isomorphic to $G$.
\end{lem}

\begin{proof}
  That $a$ is neither a loop nor an isthmus of $G$ implies that $G$ is
  the only extension $G'$ of $G\del a$ by the element $a$ with
  $G'/a=G/a$.  Therefore if $t$ is a non-loop in $\cl{K}Z\cap (X\cup
  Y)$, then $K|(Z\cup t)$ is isomorphic to $G$, for otherwise $K/t$
  could not have $G/a$ as the further contraction $K/(X\cup Y)$.  More
  generally, the same observation shows that in any contraction $K/U$
  with $U\subseteq X\cup Y$ for which $\rk{K/U}Z=\rk{K} Z$, if $t$ is a
  non-loop of $K/U$ in $\cl{K/U}Z\cap (X\cup Y)$, then $K/U|(Z\cup
  t)$ is isomorphic to $G$.  It follows that if $K/X$ does not have a
  minor isomorphic to $G$, then $\rk{K/X}Z=r(G/a)=\rk{K}Z-1$.  In this
  case, there is a subset $X'$ of $X$ and element $t\in X-X'$ for
  which $K/X'|(Z\cup t)$ is isomorphic to $G$, so $K\del Y$ has a
  minor isomorphic to $G$.
\end{proof}

In what follows, for a matroid $G$ and element $a$ of $G$, we let
$G_a$ denote the principal truncation of $G$ at $a$, that is,
$(G/a)\oplus I^*(a)$; dually, $G^a$ denotes the principal lift of $G$
at $a$, that is, $(G\del a)\oplus I(a)$.

\begin{thm}\label{thm:k4}
  Let $\mathcal{C}$ be a minor-closed class of matroids that has the
  following properties: the excluded minors for $\mathcal{C}$ have
  neither loops nor isthmuses, and whenever an excluded minor $G$ for
  $\mathcal{C}$ is written as a Higgs lift $L^j_{G_1,G_2}$ of a proper
  quotient $G_1$ and a proper lift $G_2$ of $G$, then $j=1$ and there
  is an element $a$ in $G$ such that
    \begin{enumerate}
    \item $G_1$ is the principal truncation $G_a$ and $a$ is the only
      loop of $G_a$, and
    \item $G_2$ is the principal lift $G^a$ and $a$ is the only
      isthmus of $G^a$.
    \end{enumerate}
    If $M(A)$ and $N(B)$ are matched matroids in $\mathcal{C}$, then
    $M\fsp N$ is in $\mathcal{C}$.
\end{thm}

\begin{proof}
  Let $K=M\fsp N$.  Toward getting a contradiction, assume
  $K\not\in\mathcal{C}$; say $K\del X/Y= G$ where $G$ is one of the
  excluded minors for $\mathcal{C}$.  By Corollary~\ref{cor:minor2} and
  the assumption that $\mathcal{C}$ is minor closed, we may assume
  $X\subseteq A-B$ and $Y\subseteq B-A$.  These inclusions and
  Proposition \ref{pro:hminor} give $G=K\del X/Y=L^j_{(N/Y)_0,(M\del
    X)_1}$ for some integer $j$.  If $(M\del X)_1$ were $G$, then
  either $M\del X=G$ or $G$ would have isthmuses; neither conclusion
  is possible, so $(M\del X)_1$ is a proper lift of $G$.  Similarly,
  $(N/Y)_0$ is a proper quotient of $G$.  By the hypotheses about the
  excluded minors for $\mathcal{C}$, we have $j=1$, $(M\del X)_1
  =G^a$, and $(N/Y)_0 =G_a$ for some element $a$ in $G$.  Let the
  ground set of $G\del a$ be $Z$.  Only loops and isthmuses of $(M\del
  X)_1$ and $(N/Y)_0$ can be outside $A\cap B$, so the hypotheses give
  $Z\subseteq A\cap B$.  We now consider which of $A$ and $B$ the
  element $a$ is in.

  First assume $a\in A-B$.  Set $K'=K/a$.  Since $a\not\in B$, from
  $(N/Y)_0 =G_a$ we get $N/Y =G/a$, so
  $$
K'.Z\=K/(X\cup Y\cup  a)=N/Y\=G/a.
$$ 
  Note that $K|(Z\cup a)$ is an extension of $G\del a$ since $(M\del
  X)_1 =G^a$.  The idea in the proof of Lemma~\ref{lemma:k4} shows
  that if $a$ is a non-loop of $K$ that is in $\cl{K}Z$, then, since
  $G/a$ is a minor of $K/a$, we would have $K|(Z\cup a)=G$, contrary
  to $M$ not having $G$ as a minor.  Therefore $a$ is either a loop or
  an isthmus of $M\del X$, so
  $$
K'|Z\=K\del (X\cup Y)/a\=M\del X/a\=G\del a.
$$ 
  Therefore, by Lemma~\ref{lemma:k4}, either $K'/X$ or $K'\del Y$ has
  a minor isomorphic to $G$.  Restated, this conclusion is that either
  $N$ or $M/a$ (and so $M$) has a minor isomorphic to $G$.  This
  contradiction completes the argument if $a\in A-B$.  The case $a\in
  B-A$ follows by duality.  (Using duality is justified by the
  observation that the class
  $\mathcal{C}^*=\{M^*\,:\,M\in\mathcal{C}\}$ that is dual to
  $\mathcal{C}$ satisfies the hypotheses of the theorem.)

  Finally, assume $a\in A\cap B$.  Thus, $M\del X=G^a$ and $N/Y=G_a$,
  so
  $$K\del (X\cup Y\cup a) \= M\del (X\cup a) \= G\del a$$
  and
  $$K/(X\cup Y\cup a) \= N/(Y\cup a) \= G/a.$$
  Therefore by Lemma~\ref{lemma:k4}, either $K/(X\cup a)$ or $K\del Y$
  has a minor isomorphic to $G$, that is, either $N/a$ (and so $N$) or
  $M$ has such a minor.  This contradiction completes the proof.
\end{proof}

Some familiar matroids have the properties that are hypothesized in
Theorem~\ref{thm:k4} for the excluded minors of $\mathcal{C}$.  Below
we show that any binary projective geometry $\PG(n-1,2)$ of rank $3$
or more has these properties, as does the cycle matroid of a complete
graph, $M(K_{n+1})$, for $n\geq 3$.  To mildly extend the collection of
matroids known to have these properties, note that if $M$ has these
properties, then so does $M^*$.  The arguments below use two
well-known results about quotients: if $N(E)\les M(E)$, then there is
a matroid $K$ with $K|E=M$ and $K.E=N$; also, if $N\les M$ and
$r(N)=r(M)$, then $N=M$.  The arguments also use two simple
observations: if $N\les M$, then $N|A\les M|A$ for any $A\subseteq E$;
also, a cyclic flat of size $3$ that has rank at least two is a line.

First consider the binary projective plane, $\PG(2,2)$, that is, the Fano
plane, $F_7$.  Assume $F_7 = L^j_{G_1,G_2}$ where $G_1$ and $G_2$ are,
respectively, a proper quotient and a proper lift of $F_7$.  Thus,
$r(G_1)<3<r(G_2)$.  By part (7) of Theorem \ref{thm:hcrypt}, each
line of $F_7$ must be a cyclic flat of either $G_1$ or $G_2$.  It
follows that if $r(G_1)<2$, then at least six of the lines of $F_7$
would be lines of $G_2$, which, by the structure of the lines of
$F_7$, would yield the contradiction $r(G_2)=3$.  Thus, $r(G_1)=2$.
Therefore there is a single-element extension $P'$ of $F_7$ by a
non-loop $x$ such that $G_1=P'/x$.  Since $F_7$ is modular, $P'$
extends $F_7$ by adding $x$ freely in the plane, on a line, or
parallel to a point $a$; the first two options would force all but at
most one line of $F_7$ to be lines of $G_2$, which would yield the
contradiction $r(G_2)=3$, so $G_1$ is the principal truncation
$(F_7)_a$.  Therefore, all lines of $F_7$ that do not contain $a$ must
be lines of $G_2$; this conclusion and the structure of the lines of
$F_7\del a$ imply that $G_2$ is the principal lift, $(F_7)^a$.

Now consider the cycle matroid $M(K_4)$.  Assume $M(K_4)$ is
$L^j_{G_1,G_2}$ where $G_1$ and $G_2$ are, respectively, a proper
quotient and a proper lift of $M(K_4)$.  By part (7) of Theorem
\ref{thm:hcrypt}, each cyclic flat of $M(K_4)$ must be a cyclic flat 
of at least one of $G_1$ and $G_2$.
Since $r(G_2)>3$, at most two of the four cyclic lines of $M(K_4)$ are
cyclic flats (necessarily lines) of $G_2$.  Since $M(K_4)$ is
self-dual, it follows that at most two of the four cyclic lines of
$M(K_4)$ are cyclic flats of $G_1$.  Thus, two cyclic lines of
$M(K_4)$ are cyclic flats of $G_1$ and the other two are lines of
$G_2$.  From these conclusions, it follows that there is some $a$ in
$M(K_4)$ for which $G_1$ is the principal truncation $M(K_4)_a$ and
$G_2$ is the principal lift $M(K_4)^a$.

Now consider $n\geq 4$.  Let $G$ be either $\PG(n-1,2)$ or
$M(K_{n+1})$.  Let $A$ be the ground set of $G$.  Let $G_1$ and $G_2$
be a proper quotient and a proper lift of $G$ with $G= L^j_{G_1,G_2}$.
Thus, $j>0$.  By the structure of $G$, if all lines of $G$ were lines
of $G_2$, then $r(G_2)=n$, contrary to $G_2$ being a proper lift of
$G$.  Let $\ell$ be a line of $G$ that is not a line of $G_2$ and let
$X$ be a plane of $G$ that contains $\ell$; if $G$ is $M(K_{n+1})$,
choose $X$ so that $G|X$ is $M(K_4)$.  By Proposition
\ref{pro:hminor}, $G|X = L^j_{G_1|X,G_2|X}$.  Since $\ell$ is a line
of $G|X$ but not of $G_2|X$, it follows that $G_2|X$ is a proper lift
of $G|X$; this conclusion and the inequality $j>0$ imply that $G_1|X$
is a proper quotient of $G|X$.  By the results in the previous two
paragraphs, $G_1|X$ is the principal truncation $(G|X)_a$ for some
$a\in X$, so $a$ is a loop of $G_1$.  Let $G'$ be an extension of $G$
to the set $A\cup T$ such that $G'\del T= G$ and $G'/T=G_1$.  Thus,
$\rk{G'}T=\rk{G'}{T\cup a}$.  For any line $\ell'$ of $G$ with
$a\not\in \ell'$, we have $\rk{G'}{T\cup \ell'}=\rk{G'}{T\cup \ell'\cup
a}$, so $\ell'$ is not a flat of $G_1$.  Part (7) of Theorem
\ref{thm:hcrypt} then implies that each line $\ell'$ of $G$ with
$a\not\in \ell'$ is a line of $G_2$. This conclusion and the structure
of $G\del a$ gives $r(G_2\del a)=n$, from which it follows that
$G_2\del a = G\del a$, so $G_2$ is the principal lift $G^a$.  Also,
$a$ must be the only loop of $G_1$.  By part (7) of Theorem
\ref{thm:hcrypt} and the observation that no cyclic flat of $G$ that
contains $a$ is a cyclic flat of $G^a$, it follows that all such sets
are cyclic flats of $G_1$; by considering a maximal chain of such
flats, we see that $G_1$ must have rank $n-1$.  These conclusions and
that $a$ is a loop of $G_1$ imply that $G_1$ is the principal
truncation $G_a$.

\end{document}